\documentclass[11pt,a4paper,leqno]{article}
\usepackage[T1]{fontenc}
\usepackage[utf8]{inputenc}
\usepackage[provide=*,british]{babel}
\usepackage{mlmodern}
\usepackage{microtype}
\usepackage{amsmath,amssymb,amsthm,mathtools}
\usepackage[margin=29mm]{geometry}
\usepackage{booktabs,array}
\usepackage{flafter}
\usepackage{tikz}
\usetikzlibrary{arrows.meta}
\usepackage[unicode,hidelinks]{hyperref}
\hypersetup{
  pdftitle={Sparse Moore-local realisations of binary irreducible polynomials on near-square lattice regions},
  pdfauthor={Lizhong Chen}
}
\newtheoremstyle{vintageplain}
  {0.45\baselineskip}{0.45\baselineskip}
  {\itshape}{\parindent}{\scshape}{.}{0.5em}{}
\newtheoremstyle{vintagedefinition}
  {0.45\baselineskip}{0.45\baselineskip}
  {\normalfont}{\parindent}{\scshape}{.}{0.5em}{}
\newtheoremstyle{vintageremark}
  {0.45\baselineskip}{0.45\baselineskip}
  {\normalfont}{\parindent}{\scshape}{.}{0.5em}{}
\theoremstyle{vintageplain}
\numberwithin{equation}{section}
\newtheorem{theorem}{Theorem}[section]
\newtheorem{proposition}[theorem]{Proposition}
\newtheorem{lemma}[theorem]{Lemma}
\newtheorem{corollary}[theorem]{Corollary}
\theoremstyle{vintagedefinition}
\newtheorem{definition}[theorem]{Definition}
\theoremstyle{vintageremark}

\newcommand{\F}{\mathbb F_2}
\newcommand{\one}{\mathbf1}

\newcommand{\diag}{\operatorname{diag}}
\newcommand{\coeff}{\operatorname{coeff}}
\newcommand{\ord}{\operatorname{ord}}

\newcommand{\degout}{\deg^{+}}
\newcommand{\degin}{\deg^{-}}
\DeclareMathOperator{\rank}{rank}

\allowdisplaybreaks[1]

\title{Sparse Moore-local realisations of binary\\
irreducible polynomials on near-square lattice regions}
\author{\textsc{Lizhong Chen}\\[0.4em]
  \small Department of Mathematics, The Hong Kong University of Science and Technology\\
  \small Clear Water Bay, Kowloon, Hong Kong\\
  \small \texttt{lchendh@connect.ust.hk}\\
  \small ORCID: 0009-0000-3495-7319}
\date{}

\makeatletter
\renewcommand{\maketitle}{%
  \begin{center}
    {\large\bfseries\@title\par}
    \vspace{0.9em}
    {\small\@author\par}
  \end{center}
  \vspace{0.35em}}
\renewcommand{\section}{\@startsection{section}{1}{\z@}%
  {-3.5ex \@plus -1ex \@minus -.2ex}%
  {2.3ex \@plus .2ex}{\normalfont\large\scshape}}
\renewcommand{\subsection}{\@startsection{subsection}{2}{\z@}%
  {-3.25ex \@plus -1ex \@minus -.2ex}%
  {1.5ex \@plus .2ex}{\normalfont\normalsize\itshape}}
\renewenvironment{proof}[1][\proofname]{%
  \par\pushQED{\qed}\normalfont
  \topsep6\p@\@plus6\p@\relax
  \trivlist
  \item[\hskip\labelsep\hskip\parindent
    {\scshape #1\@addpunct{.}}]\ignorespaces
}{\popQED\endtrivlist\@endpefalse}

\makeatother

\begin{document}
\maketitle

\begin{abstract}
For every \(N\ge36\), we realise any prescribed monic irreducible binary
polynomial of degree \(N\) as the characteristic polynomial of a linear
hybrid cellular automaton on a near-square region of exactly \(N\) cells.
The transition matrix is Moore-local with a null boundary and has at most
\(3N-1\) directed nonself dependencies. The dependency graph retains
a bidirectional Hamilton path; its indegree, outdegree and underlying
undirected degree are at most six. The underlying graph contains an
explicit square grid minor of side proportional to \(\sqrt N\).
The deterministic synthesis takes \(O(N^3)\) bit operations.
The construction combines a local similarity transformation with a
transport potential and joint routing across consecutive row gaps.
Exact verification of finite certificates, followed by induction,
proves the required routing inequalities for every admissible width.
We also give an entirely analytic construction
with fewer than \(7N/2\) dependencies and prove a lower bound of
\(5N/2-O(\sqrt N)\) for the retained path and full rectangular grid minor.
This lower bound is sharp when the characteristic polynomial is unrestricted.
\end{abstract}



\section{Introduction}
\label{sec:introduction}

A finite binary linear cellular automaton evolves according to a
transition matrix over \(\F\). We allow the local rules to vary from
cell to cell, as in the hybrid automata of the synthesis literature.
The characteristic polynomial determines much of the automaton's orbit
structure, while the nonzero entries of the transition matrix specify
which cells communicate.
We study the problem of prescribing an irreducible characteristic
polynomial while restricting these dependencies to a two-dimensional
lattice. Each cell may depend only on itself and on the eight
surrounding cells, giving the Moore neighbourhood. The synthesis must
also retain a two-dimensional connection structure as the number of
cells grows. We seek a realisation that meets these algebraic and
geometric requirements with exactly the prescribed number of cells.

To make the geometric requirement precise, we require the underlying
undirected dependency graph to contain a square grid minor of side
proportional to \(\sqrt N\), where \(N\) is the number of cells.
Such a minor consists of disjoint connected groups of cells, joined
according to a square grid. It gives a two-dimensional connection
structure whose scale grows with the physical region. A path cannot
satisfy this condition, irrespective of its placement. The edges used
within the groups and between adjacent groups must all be present in
the underlying undirected dependency graph of the final matrix.
This requirement concerns support in the physical cell coordinates
and is not an invariant of arbitrary linear changes of state coordinates.

For one-dimensional automata, the synthesis problem has a classical
solution. Mesirov and Sweet~\cite{MS87} proved that every irreducible
binary polynomial is the denominator of a continued fraction whose
partial quotients all have degree one. In matrix terms, this gives a
symmetric tridiagonal realisation with ones on both adjacent diagonals.
Serra and Slater~\cite{SS90} use a finite-field variant of Lanczos
tridiagonalisation to construct a similarity between an irreducible
companion matrix and a rule-\(90/150\) matrix.
Cattell's thesis~\cite{C95} and the work of Cattell and
Muzio~\cite{CM96} develop the synthesis of linear hybrid cellular
automata from irreducible polynomials. The resulting one-dimensional
automata use rules \(90\) and \(150\). Rule \(90\) adds the states of
the two neighbouring cells; rule \(150\) also adds the cell's own state.

Structured companion matrices provide another setting for realising a
prescribed characteristic polynomial. Studies of pentadiagonal forms and
nonsparse companion matrices~\cite{DFGVM19,EV16} use paired elementary
row and column operations to change matrix support through similarity.
Here we require a binary matrix whose dependencies are local in the
physical lattice and whose underlying undirected dependency graph
contains a grid minor of growing size.

A one-dimensional path can be placed on a rectangular grid in serpentine
order. The update then remains local, and its characteristic polynomial
does not change. This observation already appears in Janoowalla's
study of maximum-length two-dimensional automata~\mbox{\cite[pp.~36--37]{J92}}.
The same discussion asks whether every primitive polynomial has a
non-snake realisation, meaning one that is not obtained merely by placing
a one-dimensional automaton in serpentine order. Janoowalla's
nearest-neighbour model and restricted-rule searches use different
constraints from the Moore-local realisations considered here.
Subsequent work includes the
regular \(2\times n\) configurations of Cattell, Zhang, Serra and
Muzio~\cite{CZSM99}, and the characterisation and synthesis of a particular
hybrid two-dimensional transformation by Chattopadhyay, Choudhury and
Dihidar~\cite{CCD99}.
Zhu~\cite{Z03} studies regular rectangular configurations, derives
characteristic-polynomial formulas for \(3\times m\) automata, and
searches for maximum-period examples.

Matsumoto~\cite{M98} studies a two-dimensional family with
orthogonal-neighbour dependencies and gives arithmetic necessary conditions
for maximum period, together with examples. The family has a fixed update
pattern and restrictions on the rectangle dimensions, including coprimality.
For the
null-boundary Moore model, Choudhury, Nayak, Sahoo and
Rath~\cite{CNSR08} give matrix descriptions of uniform and hybrid
linear rules.

Recent work on linear non-uniform cellular automata studies systems
whose local maps depend on the cell. Phung~\cite{P24} develops their
duality and invertibility theory on group-indexed configuration spaces.
Another algebraic direction, surveyed by Manzoni, Mariot and
Menara~\cite{MMM26}, uses cellular automata to construct combinatorial
designs. In the linear constructions reviewed there, orthogonality of
the associated Latin squares is characterised by coprimality of the
rule polynomials. Mariot~\cite{M23} also studies maximal cycles in
systems formed from pairs of orthogonal cellular automata, using the
associated Sylvester matrix in the linear case. Our synthesis problem
prescribes the characteristic polynomial of a finite global update
while restricting its dependencies to a two-dimensional Moore
neighbourhood.

We prove that every monic irreducible binary polynomial of every degree
\(N\ge36\) has such a realisation. The cells occupy a rectangle of
complete rows, followed, when necessary, by one partial row. This gives
exactly \(N\) cells for every degree, including prime degrees, and a
bounding rectangle whose height-to-width ratio is at most \(4/3\).
The transition matrix is Moore-local with a null exterior boundary.
Its directed dependency graph retains a bidirectional spanning path,
and its underlying undirected graph contains a square grid minor
of side \(\lfloor\sqrt N/3\rfloor\). Each cell has at most six
neighbours, with indegree and outdegree at most six, and the total number
of directed nonself dependencies is at most \(3N-1\). The deterministic
construction takes \(O(N^3)\) bit operations. Its proof uses fixed finite
certificates, verified by exact computation, together with an induction
that applies to every width.

The constructions share a common starting point: a classical tridiagonal
realisation placed in serpentine order. A partial final row is attached on the side at which
the complete part of the path ends. We conjugate the path matrix \(B\)
by an involution \(I+R\), where \(R\) sends selected odd-row coordinates
to neighbouring even-row coordinates within the complete rows.
The turn into the partial row remains in an outer column, outside the
support of \(R\). Hence \(R^2=RBR=0\), and the conjugate is
\(B+RB+BR\). For the analytic construction, we select the shear
columns using two complementary masks. Within each pair of columns used for a grid-minor branch set, the
mask bits are opposite; between successive pairs they are equal.
The coefficient of each added cross-column dependency is the sum of
the adjacent mask bits. Among the paired columns, these dependencies
remain within each pair and cancel between consecutive pairs. This gives
an analytic construction with fewer than \(7N/2\) dependencies and all
three degree bounds at most six. Its grid-minor witnesses are independent
of the diagonal entries of \(B\).

To obtain the stronger bound, we let the branch sets vary between rows
and choose the masks from the diagonal differences. A transport potential measures the total displacement of the incoming
columns of the branch sets from fixed alternating columns. A uniform row lemma bounds the added connections
plus the change in this potential. Its proof uses a finite-state dynamic
program and a fixed certificate that bounds the minimum cost for every
input word. At even width, we choose the routes across two consecutive
gaps jointly, so that the outgoing interface of the first route is
the incoming interface of the second. This reduces the combined cost
while allowing either route to use more than its separate allowance.
The intermediate potentials cancel when these inequalities are summed.
A sharper estimate at the standard initial interface handles the
remaining boundary family and gives at most \(3N-1\) dependencies
for every size. Throughout, each gap contributes at most two new
neighbours per cell and supplies every link required by the full grid
minor.

The full rectangular grid minor also gives a lower bound. Any directed
graph with a bidirectional spanning path and an \(h\times k\) grid
minor in its underlying graph has at least
\(2N-2+(h-1)(k-1)\) directed edges. For our near-square parameters,
this is \(5N/2-O(\sqrt N)\). We attain this bound at the level of
support by an orthogonal-neighbour construction on the same region.
Its spectrum is not prescribed, so the optimal connection bound for
irreducible-polynomial realisation remains undetermined.

For physical widths six through eleven, a separate refinement gives
a smaller explicit bound on the connection count. It uses three fixed finite tables
whose entries are verified exhaustively. A potential inequality then
extends the finite verification to every height, every permitted partial
row and every binary path diagonal. In particular, it covers all
prescribed irreducible inputs of every degree \(36\le N\le143\).

When the prescribed polynomial is primitive, this construction gives a
maximum-period automaton without adding cells or changing the time
step. We also obtain a synthesis algorithm whose input is \(N\) and
a complete certified prime factorisation of \(2^N-1\). The arithmetic part is
a consequence of Shoup's deterministic algorithms for constructing
irreducible polynomials and searching for primitive
elements~\cite{S90,S92}; the contribution here is their combination with
the local two-dimensional realisation. We specify a recursive
primality-certificate format and include its verification in a
polynomial bound in \(N+L\), where \(L\) is the total length of the
supplied factor data. Discovering the factors and producing their
certificates are costs of preparing that input.
Statistical properties of a chosen output stream are a separate question
from the state-period result proved below.

\section{Preliminaries and main results}
\label{sec:results}

For positive integers \(h,w\), let
\[
  V_{h,w}=\{0,\ldots,h-1\}\times\{0,\ldots,w-1\}.
\]
A finite lattice region \(\Gamma\subset\mathbb Z^2\) contains one
state bit at each of its cells. A state is a column vector
\(X\in\F^\Gamma\), and a transition matrix
\(A=(A_{uv})_{u,v\in\Gamma}\) acts by \(X(t+1)=AX(t)\).
Thus \(A_{uv}=1\) represents a directed dependency \(v\to u\).
The rules are allowed to vary from cell to cell. Logarithms in the
arithmetic estimates are to base two.

\begin{definition}
A matrix \(A\) on \(\Gamma\) is \emph{Moore-local} if
\[
 A_{(r,c),(r',c')}=0
 \quad\text{whenever}\quad
 \max\{|r-r'|,|c-c'|\}>1.
\]
The boundary is \emph{null}: coordinates outside \(\Gamma\) have value
zero and are not states of the automaton. Opposite sides are not identified.
\end{definition}

Let \(D(A)\) be the directed graph of nonzero off-diagonal entries of
\(A\), with the direction convention above. Let \(U(A)\) be its simple
underlying undirected graph. Write
\[
 m(A)=|\{(u,v):u\ne v,\ A_{uv}=1\}|,\qquad
 \ell(A)=|\{v:A_{vv}=1\}|.
\]
In particular, a bidirectional connection contributes two to \(m(A)\)
and one to \(|E(U(A))|\). Degrees in \(D(A)\) exclude self-dependence.

We use the usual branch-set description of a graph minor. An
\(a\times b\) grid minor in \(U(A)\) consists of \(ab\) pairwise
disjoint connected vertex sets, indexed by an \(a\times b\) grid,
with an edge between the sets at every pair of adjacent grid positions.
All connecting edges must belong to \(U(A)\).

The characteristic polynomial of \(A\) is
\(\chi_A(x)=\det(xI-A)\). A monic irreducible polynomial \(p\in\F[x]\)
of degree \(n\) is \emph{primitive} if the residue of \(x\) in
\(\F[x]/(p)\) has multiplicative order \(2^n-1\).

We now specify the regions used in the construction. For
\(h\ge2\), \(w\ge4\) and \(0\le r<w\), put
\begin{equation}
 \Gamma_{h,w,r}=V_{h,w}\cup
 \begin{cases}
  \{(h,j):0\le j<r\},&h\text{ even},\\
  \{(h,j):w-r\le j<w\},&h\text{ odd}.
 \end{cases}
 \label{eq:region}
\end{equation}
Thus the first \(h\) rows are complete, and the possible partial row
contains \(r\) cells on the left or right according to the parity of \(h\).
For \(N\ge36\), we use the region
\begin{equation}
 \Gamma_N=\Gamma_{h,w,r},\qquad
 w=\lfloor\sqrt N\rfloor,\quad h=\lfloor N/w\rfloor,\quad r=N-hw.
 \label{eq:canonical-region}
\end{equation}
This region has exactly \(N\) cells. The constructions below bound its
aspect ratio and give explicit grid-minor witnesses.

\begin{theorem}
\label{thm:main}
Let \(N\ge36\), and let \(p\in\F[x]\) be any monic irreducible
polynomial of degree \(N\). Use \(h,w,r\) from
\eqref{eq:canonical-region}, and put \(k=\lfloor(w-2)/2\rfloor\).
There is a deterministic algorithm that constructs a transition matrix
\(A\) on \(\Gamma_N\) with the following properties:
\begin{enumerate}
  \item \(\chi_A=p\), and \(A\) is Moore-local with a null boundary;
  \item \(D(A)\) contains a bidirectional spanning path and is therefore
        strongly connected;
  \item \(U(A)\) contains an explicitly specified \(h\times k\) grid
        minor, and hence a \(t\times t\) grid minor for
        \(t=\lfloor\sqrt N/3\rfloor\);
  \item every vertex has indegree and outdegree at most six,
        \(U(A)\) has maximum degree at most six, and
        \[
          m(A)\le3N-1<3N;
        \]
  \item the algorithm uses \(O(N^3)\) bit operations and
        \(O(N^2)\) bits of space, including the output of the local
        rules and the grid-minor witnesses.
\end{enumerate}
\end{theorem}

More precisely, the construction gives \(m(A)\le3N-r-1\) when
\(w\) is odd and \(m(A)\le3N-r-2\) when \(w\) is even and
\(h\le w+1\). The remaining case is \(h=w+2\), \(r=0\),
with even \(w\), for which the bound is \(3N-1\).
The proof is computer-assisted: finite certificates establish the
single-gap, joint two-gap and standard-start inequalities in
Section~\ref{sec:transport}. Their induction proofs appear in
Appendix~\ref{app:transport}. The algorithm takes the coefficients of
\(p\) as input and assumes that \(p\) is irreducible.

We first establish a separate analytic construction, which also gives
explicit formulas for its connection count.

\begin{theorem}
\label{thm:paired}
Under the hypotheses of Theorem~\ref{thm:main}, a deterministic
paired-shear construction on \(\Gamma_N\) realises \(p\), retains the
bidirectional spanning path and full \(h\times k\) grid minor, is
Moore-local with a null boundary, and has all three nonself degree bounds
at most six. It satisfies \(m(A)<7N/2\) and uses \(O(N^3)\) bit
operations and \(O(N^2)\) bits of space.
\end{theorem}

The proof of Theorem~\ref{thm:paired} is entirely analytic. Its geometric
construction works on every region in \eqref{eq:region}, for any binary
path diagonal. The transport construction proving Theorem~\ref{thm:main}
also accepts every path diagonal, with \(w\ge6\). We develop the common
local support calculation first, then prove the two connection bounds.

\begin{corollary}
\label{cor:primitive}
If \(p\) in Theorem~\ref{thm:main} is primitive, the resulting
automaton has at most \(3N-1\) directed nonself dependencies, the
zero state is fixed, and all \(2^N-1\) nonzero states belong to a
single cycle.
\end{corollary}

\begin{corollary}
\label{cor:factors}
Let \(N\ge36\). Given a complete prime factorisation of \(2^N-1\)
with certificates in the format of Appendix~\ref{app:certificates},
a deterministic algorithm verifies the supplied data and constructs a
matrix satisfying all the geometric conclusions and the connection
bound \(m(A)\le3N-1\) of Theorem~\ref{thm:main}, and the orbit conclusion of
Corollary~\ref{cor:primitive} using at most \(K(N+L)^{16}\) bit operations,
where \(L\) is the total encoding length of the factor data and certificates
and \(K\) is an absolute constant. Invalid data are rejected within
the same bound.
\end{corollary}

Corollary~\ref{cor:factors} includes reading and checking the input,
as well as synthesising the automaton. The arithmetic arguments for
both corollaries are given in Section~\ref{sec:period}. The cubic
bound in Theorem~\ref{thm:main} starts with the supplied polynomial.

\section{The local shear construction}
\label{sec:shear}

\subsection{A path realisation}

For a binary word \(d=(d_0,\ldots,d_{n-1})\), let \(T(d)\) be the
symmetric tridiagonal matrix with diagonal \(d\) and ones on both
adjacent diagonals. It describes a bidirectional path with a choice of
self-dependence at each vertex.

\begin{proposition}
\label{prop:seed}
Given a monic irreducible polynomial \(p\in\F[x]\) of degree \(n\ge2\),
one can deterministically construct a word \(d\in\F^n\) such that
\(\chi_{T(d)}=p\), using \(O(n^3)\) bit operations and \(O(n^2)\)
bits of space.
\end{proposition}

We use the linear system of Mesirov and Sweet~\cite{MS87}, with a fixed
pivot order, setting all free variables to zero. Appendix~\ref{app:seed} gives the
algorithm and a complete proof, including the passage from its linear
system to a tridiagonal matrix. We now apply a local similarity
transformation that preserves the characteristic polynomial and adds
dependencies so that the underlying undirected graph contains a large
grid minor.

\subsection{Serpentine placement and the shear}

Throughout Sections~\ref{sec:shear} and~\ref{sec:minor}, let
\(h\ge2\), \(w\ge4\), \(0\le r<w\) and \(N=hw+r\).
Place a word \(d\in\F^N\) on \(\Gamma_{h,w,r}\) in serpentine order:
traverse each even complete row from left to right and each odd
complete row from right to left. When \(r>0\), continue into row \(h\)
in the same parity-dependent direction. The partial row starts in
column zero if \(h\) is even, and in column \(w-1\) if \(h\) is odd.
It is therefore attached at the end of the complete-row path.

This order visits every occupied cell exactly once. All horizontal
connections join consecutive cells in a row, and every turn between
rows lies in column zero or column \(w-1\), including the turn into
the partial row. The cases \(r=0\) and \(r=1\) simply give, respectively,
no partial row and a final endpoint with no horizontal connection.
Write \(d_{i,j}\) for the diagonal bit at the physical cell \((i,j)\).
\par\noindent\begin{minipage}{\linewidth}
The placed path matrix \(B\) satisfies
\begin{equation}
  B=H+V+D,
  \label{eq:decomp}
\end{equation}
where \(H\) contains all bidirectional horizontal edges,
\(V\) contains the bidirectional turns between successive occupied
rows, and \(D=\diag(d_{i,j})\). The possible horizontal connections
within the partial row belong to \(H\). The matrix \(B\) is a permutation
conjugate of \(T(d)\).
\end{minipage}\par

\begin{figure}[htbp]
\centering
\begin{tikzpicture}[x=0.55cm,y=0.43cm,
 cell/.style={circle,fill=black,inner sep=1.4pt},
 path/.style={black!85,thick}]
 \begin{scope}
  \foreach \i in {0,...,5} {
   \draw[path] (0,-\i)--(5,-\i);
   \foreach \j in {0,...,5} \node[cell] at (\j,-\i) {};
  }
  \foreach \i in {0,...,4} {
   \pgfmathtruncatemacro{\ii}{\i+1}
   \pgfmathtruncatemacro{\turn}{5*(1-mod(\i,2))}
   \draw[path] (\turn,-\i)--(\turn,-\ii);
  }
  \fill[black!15] (-0.3,-6.3) rectangle (0.3,-5.7);
  \draw[path] (0,-5)--(0,-6);
  \node[cell] at (0,-6) {};
  \node at (2.5,-8.3) {\(N=37,\ h=6,\ w=6,\ r=1\)};
 \end{scope}
 \begin{scope}[xshift=8cm]
  \foreach \i in {0,...,6} {
   \draw[path] (0,-\i)--(5,-\i);
   \foreach \j in {0,...,5} \node[cell] at (\j,-\i) {};
  }
  \foreach \i in {0,...,5} {
   \pgfmathtruncatemacro{\ii}{\i+1}
   \pgfmathtruncatemacro{\turn}{5*(1-mod(\i,2))}
   \draw[path] (\turn,-\i)--(\turn,-\ii);
  }
  \fill[black!15] (3.7,-7.3) rectangle (5.3,-6.7);
  \draw[path] (5,-6)--(5,-7)--(4,-7);
  \node[cell] at (4,-7) {};
  \node[cell] at (5,-7) {};
  \node at (2.5,-8.3) {\(N=44,\ h=7,\ w=6,\ r=2\)};
 \end{scope}
\end{tikzpicture}
\caption{The partial row is attached on the side at which the
complete-row path ends. Dark lines denote bidirectional path
connections; shaded cells form the partial row. The shear acts only
between complete rows. The \(N=44\) diagram illustrates the geometry,
independently of any choice of input polynomial.}
\label{fig:fringe}
\end{figure}

Put \(k=\lfloor(w-2)/2\rfloor\). Define masks
\(a^+,a^-\in\F^w\), zero outside columns \(1,\ldots,2k\), by
\[
 \begin{aligned}
 (a^+_{2b+1},a^+_{2b+2})&=(b\bmod2,1-(b\bmod2)),\\
 (a^-_{2b+1},a^-_{2b+2})&=(1-(b\bmod2),b\bmod2)
 \qquad(0\le b<k).
 \end{aligned}
\]
The columns are paired as \(\{2b+1,2b+2\}\), \(0\le b<k\).
In each complete row, the two cells in such a pair will form one branch
set of the grid minor. The upward mask starts \(0,1,1,0,0,1,1,0,\ldots\) on the paired
columns; the downward mask is its complement there. Both vanish at
the outer columns and at any unpaired internal column.
For each odd complete row \(0\le i<h\), define \(R\) by
\begin{align}
 (i,j)&\longrightarrow(i-1,j)
 &&\text{if \(a^+_j=1\)},\notag\\
 (i,j)&\longrightarrow(i+1,j)
 &&\text{if \(a^-_j=1\) and \(i+1<h\)}.
 \label{eq:shear}
\end{align}
Every arrow has coefficient one, and all other entries of \(R\) are zero.
Since \(2k\le w-2\), both endpoints lie in complete rows and internal
columns. No arrow of \(R\) has an endpoint in the partial row.
Let \(\mathcal R\) be the set of these arrows, written as ordered pairs
\((u,v)\) with source \(u\) and destination \(v\).

\begin{lemma}
\label{lem:nil}
The matrices in \eqref{eq:decomp} and \eqref{eq:shear} satisfy
\[
 R^2=0,\qquad RV=VR=0,\qquad RHR=RDR=0.
\]
Consequently, \(S=I+R\) is an involution, and
\begin{equation}
 A=SBS=B+RH+HR+RD+DR
 \quad\text{satisfies}\quad \chi_A=\chi_B.
 \label{eq:similarity}
\end{equation}
\end{lemma}

\begin{proof}
The matrix \(R\) maps odd-row coordinates to even-row coordinates
and vanishes on the latter. This proves \(R^2=0\).
Both \(H\) and \(D\) preserve row parity, including on the partial
row, so \(RHR=RDR=0\).
The matrix \(V\) is supported on columns zero and \(w-1\), whereas every
source and destination of \(R\) is in an internal column.
Hence \(RV=VR=0\), and therefore \(RBR=0\).
Over \(\F\), we have \(S^2=I\); expanding \(SBS\) now gives
\eqref{eq:similarity}.
\end{proof}

\subsection{The nonzero entries}

We next determine the support of \(A\). This also verifies locality:
similarity alone would not prevent nonlocal entries.

\begin{lemma}
\label{lem:support}
Consider a gap between consecutive complete rows, with odd source row
\(o\), even destination row \(v\), and mask \(a=a^+\) or \(a^-\)
according to whether \(v=o-1\) or \(v=o+1\).
For \(0\le j<w-1\), both entries
\begin{equation}
 (o,j)\longrightarrow(v,j+1),\qquad
 (o,j+1)\longrightarrow(v,j)
 \label{eq:four}
\end{equation}
have coefficient \(a_j+a_{j+1}\) in \(A\). The only other possible
added entry at column \(j\) is
\begin{equation}
 (o,j)\longrightarrow(v,j)
 \quad\text{with coefficient}\quad
 a_j(d_{o,j}+d_{v,j}).
 \label{eq:optional}
\end{equation}
After these additions in \(\F\), all listed nonzero entries are
distinct. They are precisely the entries added to \(B\), and no entry
of \(B\) is removed.
\end{lemma}

\begin{proof}
For the first entry in \eqref{eq:four}, the product \(RH\) contributes
\(a_{j+1}\) and \(HR\) contributes \(a_j\). The second entry has
the same two contributions in reverse order. Thus these arrows are
present exactly where the mask changes between adjacent columns.
In particular, the contributions cancel when the adjacent mask bits
are equal. The terms \(RD\) and \(DR\) give the coefficient in
\eqref{eq:optional}. Lemma~\ref{lem:nil} excludes any further terms.

Different gaps join different pairs of rows. Within one gap, distinct
transitions give distinct cross-column entries, and none is vertical.
The horizontal and diagonal entries of \(B\) stay within a row, while
its inter-row entries lie in outer columns. Every new vertical entry
is internal, and every new cross-column entry changes the column.
Hence no added entry coincides with an entry of \(B\).
\end{proof}

Every added entry joins adjacent complete rows and changes the column
by at most one. Its endpoints therefore exist in the region. The partial
row and its boundary turn are unchanged, including when \(r=1\).
Thus \(A\) is Moore-local with a null boundary, and its retained
bidirectional spanning path makes \(D(A)\) strongly connected.
The coefficient calculation also applies to an arbitrary mask in each
complete-row gap, provided its outer bits are zero and every shear
points from an odd row to an even row. We will use this observation
for both the transport construction and the finite-width refinement.

\begin{figure}[htbp]
\centering
\begin{tikzpicture}[x=0.85cm,y=1.0cm,>=Stealth,
  cell/.style={circle,fill=black,inner sep=2pt},
  base/.style={gray,<->,thin},
  added/.style={->,thick,black!85}]
  \begin{scope}
    \foreach \y in {0,1.5} {
      \foreach \x in {0,...,3} \node[cell] at (\x,\y) {};
      \foreach \x in {0,...,2}
        \draw[base] (\x+0.08,\y)--(\x+0.92,\y);
    }
    \foreach \x in {1,2}
      \draw[added] (\x,0.10)--(\x,1.40);
    \foreach \x/\bit in {0/0,1/1,2/1,3/0}
      \node at (\x,-0.35) {\(\bit\)};
    \node[left] at (-0.18,1.5) {even row};
    \node[left] at (-0.18,0) {odd row};
    \node at (1.5,-0.85) {(a) Adjacent upward shears};
  \end{scope}
  \begin{scope}[xshift=6.5cm]
    \foreach \y in {0,1.5} {
      \foreach \x in {0,...,3} \node[cell] at (\x,\y) {};
      \foreach \x in {0,...,2}
        \draw[base] (\x+0.08,\y)--(\x+0.92,\y);
    }
    \foreach \x in {0,2} {
      \draw[added] (\x+0.04,0.08)--(\x+0.96,1.42);
      \draw[added] (\x+0.96,0.08)--(\x+0.04,1.42);
    }
    \foreach \x in {1,2}
      \draw[->,dashed,thick] (\x,0.10)--(\x,1.40);
    \node at (1.5,-0.85) {(b) Dependencies after cancellation};
  \end{scope}
\end{tikzpicture}
\caption{The mask \(0,1,1,0\) on two consecutive pairs of columns.
Grey edges belong to the bidirectional path. The dark cross-column
arrows remain at mask transitions; the cross-column entries between the
middle columns cancel.
Each dashed vertical arrow occurs only when the diagonal bits at its
endpoints differ. A downward gap has the reflected pattern.}
\label{fig:shear}
\end{figure}
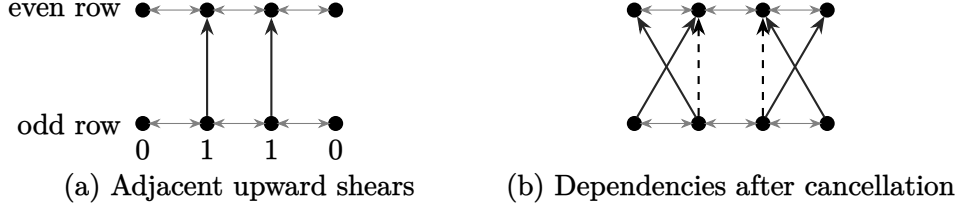

\section{Grid minors and synthesis bounds}
\label{sec:minor}

\subsection{A grid minor independent of the diagonal}

Put \(k=\lfloor(w-2)/2\rfloor\). For \(0\le i<h\) and \(0\le b<k\), define
\begin{equation}
 \mathcal B_{i,b}=\{(i,2b+1),(i,2b+2)\}.
 \label{eq:branch}
\end{equation}
These sets lie in the complete rows and avoid the two outer columns.

\begin{proposition}
\label{prop:minor}
For every diagonal word \(d\), the sets in \eqref{eq:branch} form
the branch sets of an \(h\times k\) grid minor in \(U(A)\).
\end{proposition}

\begin{proof}
The sets are pairwise disjoint, and each is connected by a retained
horizontal edge. Consecutive sets in the same row are joined by
\[
 \{(i,2b+2),(i,2b+3)\}.
\]
For consecutive rows, use the edge
\begin{equation}
 \{(i,2b+2),(i+1,2b+1)\}.
 \label{eq:crossbranch}
\end{equation}
Both masks change between columns \(2b+1\) and \(2b+2\).
Hence \eqref{eq:four} gives the arrow
\((i+1,2b+1)\to(i,2b+2)\) if \(i\) is even, and
\((i,2b+2)\to(i+1,2b+1)\) if \(i\) is odd.
Each has coefficient one after cancellation. The inequality
\(2b+2\le w-2\) ensures that the whole pair lies in internal
columns, for odd as well as even widths.
The required edges therefore exist independently of \(d\).
Contracting each branch set and deleting surplus edges gives the
claimed grid.
\end{proof}

This minor is specified by \(2hk\) vertices, \(hk\) edges within
the branch sets, and \(h(k-1)+(h-1)k\) edges between branch sets.
The partial row is outside these branch sets but belongs to the
retained spanning path, and hence to the same strongly connected system.

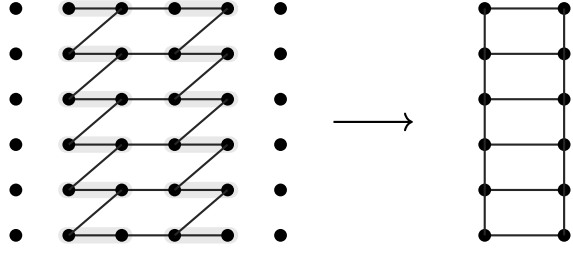
\begin{figure}[htbp]
\centering
\begin{tikzpicture}[x=0.70cm,y=0.60cm,
  cell/.style={circle,fill=black,inner sep=1.7pt},
  witness/.style={thick,black!85}]
  \foreach \r in {0,...,5} {
    \foreach \j in {0,1} {
      \pgfmathtruncatemacro{\ca}{2*\j+1}
      \pgfmathtruncatemacro{\cb}{2*\j+2}
      \fill[black!9,rounded corners=3pt]
        (\ca-0.20,-\r-0.18) rectangle (\cb+0.20,-\r+0.18);
      \draw[witness] (\ca,-\r)--(\cb,-\r);
    }
    \draw[witness] (2,-\r)--(3,-\r);
    \foreach \c in {0,...,5} \node[cell] at (\c,-\r) {};
  }
  \foreach \r in {0,...,4} {
    \pgfmathtruncatemacro{\rr}{\r+1}
    \draw[witness] (2,-\r)--(1,-\rr);
    \draw[witness] (4,-\r)--(3,-\rr);
  }
  \draw[->,thick] (6,-2.5)--(7.5,-2.5);
  \begin{scope}[xshift=6.2cm]
    \foreach \r in {0,...,5} {
      \draw[witness] (0,-\r)--(1.5,-\r);
      \node[cell] at (0,-\r) {};
      \node[cell] at (1.5,-\r) {};
    }
    \draw[witness] (0,0)--(0,-5);
    \draw[witness] (1.5,0)--(1.5,-5);
  \end{scope}
\end{tikzpicture}
\caption{The \(6\times2\) minor in a \(6\times6\) physical grid.
Each shaded pair is one branch set. Only the witness edges in the
underlying undirected graph are shown. Their presence is independent
of the diagonal word and of the optional vertical edges.}
\label{fig:minor}
\end{figure}

\subsection{Edges and degrees}

Each mask selects one column in every branch pair, so
\begin{equation}
 e=|\mathcal R|=(h-1)k.
 \label{eq:shearcount}
\end{equation}
Define the integer
\[
 \nu(d)=\sum_{(u,v)\in\mathcal R}(d_u\mathbin{\oplus}d_v),
 \qquad
 \delta=\one_{\{h\text{ even and }k\text{ even}\}},
\]
where \(\oplus\) denotes binary addition, interpreted as zero or one
before summation.

\begin{proposition}
\label{prop:counts}
The construction satisfies
\begin{align}
 m(A)&=2N-2+2(h-1)(k+1)-2\delta+\nu(d),\label{eq:edgecount}\\
 \ell(A)&=\sum_{u\in\Gamma_{h,w,r}}d_u,\notag\\
 |E(U(A))|&=N-1+2(h-1)(k+1)-2\delta+\nu(d),\notag
\end{align}
with \(0\le\nu(d)\le(h-1)k\). In particular,
\[
 m(A)\le2N-2+(h-1)(3k+2)-2\delta<\frac72N.
\]
Every vertex has indegree and outdegree at most six, and \(U(A)\)
also has maximum degree at most six.
\end{proposition}

\begin{proof}
Within the paired interval, a mask changes between the two columns of
each pair and stays constant between consecutive pairs. Its only further
possible transitions are between columns zero and one, and between
columns \(2k\) and \(2k+1\). Their indicators are \(a_1\) and \(a_{2k}\).
Thus each mask has \(k+a_1+a_{2k}\) transitions.
If \(k\) is odd, exactly one of \(a_1,a_{2k}\) is one for each mask.
If \(k\) is even, both are zero in the upward mask and both are one
in the downward mask. There are
\(\lfloor h/2\rfloor\) upward gaps and
\(\lfloor(h-1)/2\rfloor\) downward gaps. The total number of
transitions is therefore
\[
 T=(h-1)(k+1)-\delta.
\]
The path contributes \(2N-2\) directed nonself entries. Each transition
adds two entries in \eqref{eq:four}, and exactly \(\nu(d)\) entries
in \eqref{eq:optional} are present. Lemma~\ref{lem:support} shows that
they are distinct and that the diagonal is unchanged.
All added arrows point from odd rows to even rows, so none is the
reverse of another. This proves the three formulas.

Using \(N=hw+r\) and \(w\ge2k+2\), the difference between
\(7N/2\) and the stated upper bound is at least
\[
 \frac32N+2-(h-1)(3k+2)+2\delta
 \ge h+3k+4+\frac32r+2\delta>0.
\]
This gives the strict connection bound.

A cell in a paired column has one cross-column neighbour across the
internal edge of its pair in each adjacent complete-row gap. There
are at most two such gaps. At columns one and \(2k\), an outer
transition may add one further neighbour. The masks in the two
adjacent gaps are complementary at that column, so this happens in
at most one gap. Thus a paired-column cell has at most three new
cross-column neighbours. An unpaired column is adjacent to at most
one end of the paired interval and has at most one such neighbour.

For a row with two adjacent complete-row gaps, the masks in those gaps
are complementary on columns \(1,\ldots,2k\) and both vanish elsewhere.
Hence each cell is incident with at
most one shear arrow, so the shear arrows form a matching when their
directions are ignored. Each cell is therefore incident with at most
one optional vertical edge. Together with at most two path neighbours, this gives
maximum degree at most \(2+3+1=6\) in \(U(A)\). Each directed degree is bounded by the
undirected degree. Cells in the partial row have only their path
connections.
\end{proof}

The bounds on \(\nu(d)\) are attained within the class of all diagonal
words, since the shear pairs form a matching. These attainment
statements impose no irreducibility assumption. Exact per-cell degree
formulas are recorded in Appendix~\ref{app:degrees}.

\subsection{The synthesis algorithm}

\begin{proof}[Proof of Theorem~\ref{thm:paired}]
Choose \(h,w,r\) as in \eqref{eq:canonical-region}.
Since \(w^2\le N<(w+1)^2\), we have
\[
 w\ge6,\qquad h\ge w,\qquad
 H_N=h+\one_{\{r>0\}}=\lceil N/w\rceil\le w+2.
\]
The bounding rectangle therefore has aspect ratio
\[
 1\le H_N/w\le1+2/w\le4/3.
\]
Moreover, with \(k=\lfloor(w-2)/2\rfloor\),
\[
 2\le t=\lfloor\sqrt N/3\rfloor=\lfloor w/3\rfloor\le k\le h.
\]
The equality of the two floors follows from
\(w\le\sqrt N<w+1\), and the inequality follows from
\(w/3\le(w-2)/2\) for \(w\ge6\).

Apply Proposition~\ref{prop:seed} to obtain \(T(d)\), and place it
in serpentine order on \(\Gamma_{h,w,r}\). Define \(R\) by
\eqref{eq:shear} and output the support given in
Lemma~\ref{lem:support}. Lemma~\ref{lem:nil} proves that the
resulting matrix has characteristic polynomial \(p\).
Locality and strong connectivity follow from
Lemma~\ref{lem:support}; the minor and degree conclusions follow
from Propositions~\ref{prop:minor} and~\ref{prop:counts}.
Restrict the rectangular minor to its first \(t\) rows and first
\(t\) columns to obtain the square witness.

The seed algorithm uses \(O(N^3)\) bit operations and \(O(N^2)\)
bits of space. The remaining construction lists \(O(N)\)
matrix entries and \(O(N)\) minor-witness entries.
With binary coordinate indices, the output has
\(O(N\log N)\) bits. Sorting these entries costs at most
\(O(N\log^2N)\) bit operations using a comparison sort.
Even a direct implementation that searches an \(O(N)\)-entry
list at each insertion costs \(O(N^2\log N)\), within the total
cubic bound. The same bounds cover writing a dense matrix if
desired. Each local rule is a nine-bit vector.
\end{proof}

For squares of even side length, the edge bound can be sharpened.

\begin{corollary}
\label{cor:squares}
If \(N=s^2\) with even \(s\ge6\), the paired-shear algorithm in
Theorem~\ref{thm:paired} uses the complete \(s\times s\) square and
satisfies
\begin{equation}
 m(A)\le\frac72N-\frac52\sqrt N-1-2\delta,
 \qquad \delta=\one_{\{s\equiv2\pmod4\}}.
 \label{eq:squarecount}
\end{equation}
\end{corollary}

\begin{proof}
Here \(h=w=s\), \(r=0\), \(k=(s-2)/2\) and
\(e=(s-1)(s-2)/2\). Since \(s\) is even, \(k\) is even exactly
when \(s\equiv2\pmod4\). Equation~\eqref{eq:edgecount} gives
\[
 m(A)=3N-\sqrt N-2-2\delta+\nu(d)
 \le\frac72N-\frac52\sqrt N-1-2\delta.
\]
\end{proof}

The similarity transformation is used to synthesise the matrix.
Once the local rules have been written, each time step applies
\(A\) directly in the physical cell coordinates. The automaton
has exactly \(N\) state bits and uses no intermediate global
coordinate conversion.

\subsection{A sharp lower bound for the support}

The retained path and grid minor impose a connection cost independently
of the characteristic polynomial.

\begin{theorem}
\label{thm:lower}
Let \(D\) be a simple directed graph without loops on \(N\) vertices.
Suppose that \(D\) contains a bidirectional spanning path and that
its underlying simple graph contains an \(h\times k\) grid minor.
Then
\begin{equation}
 |E(D)|\ge2N-2+(h-1)(k-1).
 \label{eq:lower}
\end{equation}
For the parameters of Theorem~\ref{thm:main}, equality is attained by
an orthogonal-neighbour support on \(\Gamma_N\) with indegree and
outdegree at most three and underlying maximum degree at most four.
\end{theorem}

\begin{proof}
For a finite simple graph \(G\), write
\(\beta(G)=|E(G)|-|V(G)|+c(G)\), where \(c(G)\) is its number
of components. This cycle rank cannot increase under taking a minor.
Deleting a bridge preserves it, while deleting another edge decreases
it by one. Removing an isolated vertex also preserves it.
Contracting an edge with \(t\) common neighbours removes \(1+t\)
edges and one vertex after parallel edges are merged, and hence
decreases the cycle rank by \(t\). These operations also cover
vertex deletion. The \(h\times k\) grid has cycle rank
\((h-1)(k-1)\).

Let \(U\) be the underlying graph of \(D\), and let \(b\) count
the unordered pairs carrying both directions. Then
\(|E(D)|=|E(U)|+b\). The spanning path makes \(U\) connected and
gives \(b\ge N-1\), while the minor gives
\(|E(U)|\ge N-1+(h-1)(k-1)\). This proves \eqref{eq:lower}.

For equality on \(\Gamma_N\), retain the bidirectional serpentine
path. Here \(k\ge2\). Partition each complete row into the intervals
\[
 I_b=\{b\}\quad(0\le b<k-1),\qquad
 I_{k-1}=\{k-1,\ldots,w-1\},
\]
and choose representatives \(c_b=b\) for \(b<k-1\) and
\(c_{k-1}=w-1\). Across each complete-row gap, join the cells in
each representative column. Both outer columns are represented, so
exactly one of these \(k\) edges is already the path turn.
Orient the remaining \(k-1\) new edges downwards. The resulting
directed count is exactly \(2N-2+(h-1)(k-1)\).

The sets \(\{(i,j):j\in I_b\}\) form connected, disjoint branch
sets. Their horizontal adjacencies are path edges, and their vertical
adjacencies are the representative-column edges. Thus they certify the
full \(h\times k\) grid minor. An internal-column cell has at most
two horizontal path neighbours and one added incoming and one added
outgoing edge. An outer-column cell has at most one horizontal neighbour
and one vertical neighbour on each side. Its path turn is on at most
one side, so each directed degree is at most three. The partial row
retains only its path edges. All edges are orthogonal-neighbour edges,
giving underlying maximum degree at most four.
\end{proof}

For the near-square parameters, \(w\le h\le w+2\), \(r<w\)
and \(k=w/2+O(1)\). Consequently the lower bound is
\(5N/2-O(\sqrt N)\). Under these path and grid-minor requirements, the leading coefficient
in a uniform linear upper bound on the number of directed dependencies
is therefore at least \(5/2\).
The equality construction specifies the support, but does not show how
to choose the diagonal to realise every supplied irreducible polynomial.
Thus attaining this edge bound without prescribing the characteristic
polynomial does not determine the optimal bound when that polynomial
is prescribed.

\subsection{A uniform bound from a transport potential}
\label{sec:transport}

We now allow the branch sets to vary from row to row and choose the
shear mask from the diagonal differences in each gap. The local support
formulas remain those of Lemma~\ref{lem:support}. To control the total
number of connections, we measure the displacement of the columns at
which the branch sets meet the preceding gap.

For \(w\ge6\), put \(k=\lfloor(w-2)/2\rfloor\). For a binary word
\(u\) of length \(w\), define its prefix balance by
\[
 B_j(u)=2\sum_{t<j}u_t-j\qquad(0\le j\le w).
\]
Let \(\mathcal L_w\) be the set of words with zero first and last
bits, exactly \(k\) ones, and \(-3\le B_j(u)\le3\) at every prefix.
If the one positions are \(u_0^*<\cdots<u_{k-1}^*\), define the
\emph{transport potential}
\begin{equation}
 D(u)=\sum_{b=0}^{k-1}|u_b^*-(2b+1)|.
 \label{eq:transport-potential}
\end{equation}
Thus \(D(u)\) is the total displacement from the reference columns
\(1,3,\ldots,2k-1\). It is nonnegative and can be computed one
column at a time, since
\begin{equation}
 D(u)=\sum_{j=1}^{w-1}u_j|B_j(u)+1|.
 \label{eq:transport-local}
\end{equation}
Indeed, immediately before the one at \(u_b^*\), the balance is
\(2b-u_b^*\). Write \(s_w\) for the reference word. It belongs to
\(\mathcal L_w\) and has potential zero for either parity of \(w\).
The definition permits runs of three equal bits; at odd width the
reference word ends in three zeros.

The following row lemma supplies the masks and the connections between
successive branch sets. Its finite-state proof is given in
Appendix~\ref{app:transport}; the finite inequalities in that proof
are verified by exact integer computation.

\begin{lemma}
\label{lem:transport-row}
Let \(w\ge6\), \(x\in\mathcal L_w\), and \(q\in\F^w\).
There are binary words \(a,y,z\) of length \(w\), with
\(a_0=a_{w-1}=0\), \(z\in\mathcal L_w\), and exactly \(k\)
ones in \(y\), having the following properties. Put
\[
 \tau_j=a_j\mathbin\oplus a_{j+1}\quad(0\le j<w-1),\qquad
 \tau_{-1}=\tau_{w-1}=0,
 \qquad
 c(a,q)=2\sum_{j=0}^{w-2}\tau_j+\sum_{j=0}^{w-1}a_jq_j,
\]
where the sums are ordinary integer sums. Write
\(\xi_b,\eta_b,\zeta_b\) for the ordered one positions of
\(x,y,z\), respectively. Then
\begin{enumerate}
 \item \(\max(\xi_b,\eta_b)<\min(\xi_{b+1},\eta_{b+1})\)
       for \(0\le b<k-1\);
 \item each link \((\eta_b,\zeta_b)\) has displacement at most one;
       a same-column link has \(a_{\eta_b}q_{\eta_b}=1\), and a
       displacement-one link has
       \(\tau_{\min(\eta_b,\zeta_b)}=1\);
 \item \(\tau_{j-1}+\tau_j+a_jq_j\le2\) at every column, and
       \begin{equation}
        c(a,q)+D(z)\le2k+4+D(x).
        \label{eq:transport-row}
       \end{equation}
\end{enumerate}
The endpoints of \(y\) are unrestricted. A uniform deterministic
algorithm computes these words in \(O(w\log w)\) bit operations
and \(O(w\log w)\) bits of space.
\end{lemma}

Here \(x\) and \(y\) specify the incoming and outgoing columns of
the branch sets in one row, and \(z\) specifies the incoming columns
in the next row. Incoming and outgoing refer to the order of the rows
in the minor construction; the links between branch sets are undirected. The first condition lets us join the two columns
of each branch set by a horizontal interval, without intersecting the
other branch sets. The second condition uses precisely the links supplied
by Lemma~\ref{lem:support}. The cost \(c(a,q)\) counts all added directed dependencies, including
those not used by the grid-minor links.

We call \((a,y,z)\) a \emph{legal route} for \((x,q)\) if it
satisfies these endpoint, cardinality, interval, link and incidence
conditions, without requiring the cost inequality
\eqref{eq:transport-row}. This distinction will allow us to minimise
the combined cost of two consecutive routes.
Each mask transition supplies two directed diagonal dependencies,
although the ordered links can use at most one of them: using both
would exchange two consecutive labels. If \(K_v\) vertical links
are selected, \(U_v\) vertical dependencies are unused, and \(U_d\)
mask transitions supply no selected link, then
\[
 c(a,q)-2k=2U_d+U_v-K_v.
\]
Thus a vertical link costs one connection less than a selected
diagonal link, whereas unused support is still counted in full.
The potential records how the chosen columns affect the next gap.

\begin{proposition}
\label{prop:transport-lift}
Let \(h\ge2\), \(w\ge6\), \(0\le r<w\), and \(N=hw+r\).
For every binary path diagonal \(d\), there is a deterministic
construction on \(\Gamma_{h,w,r}\) similar to \(T(d)\), with
Moore locality, a null boundary, the retained bidirectional spanning
path, all three nonself degree bounds at most six, and an explicit
\(h\times k\) grid minor. It satisfies
\begin{equation}
 m(A)\le2N-2+(h-1)(2k+4).
 \label{eq:transport-shape}
\end{equation}
The construction uses Lemma~\ref{lem:transport-row} and, after the
path diagonal is supplied, takes \(O(N^2)\) bit operations and
\(O(N^2)\) bits of space.
\end{proposition}

\begin{proof}
Place the path as in Section~\ref{sec:shear}. Start with the word
\(x^{(0)}\) whose ones are at \(1,3,\ldots,2k-1\).
For gap \(i\), put
\(q^{(i)}_j=d_{i,j}\mathbin\oplus d_{i+1,j}\), apply
Lemma~\ref{lem:transport-row} to \((x^{(i)},q^{(i)})\), and write
the result as \((a^{(i)},y^{(i)},z^{(i)})\). Set
\(x^{(i+1)}=z^{(i)}\). All these words are read in increasing
physical column order, even in rows traversed from right to left by
the spanning path. Thus the output word in one gap is the identical
input word in the next, with the same labels and potential.

Put a shear at each selected column of each complete-row gap, directed
from its odd row to its even row. The shears need not form a matching.
Nevertheless, row parity and the zero outer mask bits give
\(R^2=RHR=RDR=RV=VR=0\), exactly as in Lemma~\ref{lem:nil}.
Consequently \(A=(I+R)B(I+R)\) is similar to \(B\).
The support calculation in Lemma~\ref{lem:support} applies separately
in each gap. It preserves the path and the diagonal, and gives only
Moore-local additions. No shear meets the partial row.

Each gap contributes at most two new neighbours to a cell, by the
third condition of Lemma~\ref{lem:transport-row}. A cell meets at
most two complete-row gaps and has at most two path neighbours.
Its undirected degree, and hence each directed degree, is at most six.
Since the added entries are distinct and do not overlap the path entries,
the total number of directed nonself dependencies is
\begin{equation}
 m(A)=2N-2+\sum_{i=0}^{h-2}c(a^{(i)},q^{(i)}).
 \label{eq:route-count}
\end{equation}
Summing \eqref{eq:transport-row} cancels the intermediate potentials.
The initial potential is zero and the final potential is nonnegative,
so
\[
 \sum_{i=0}^{h-2}c(a^{(i)},q^{(i)})
 \le(h-1)(2k+4),
\]
which proves \eqref{eq:transport-shape}.

For the minor, join the \(b\)-th one of \(x^{(i)}\) to the
\(b\)-th one of \(y^{(i)}\) by their horizontal interval in row
\(i<h-1\). In the last row use the singleton at the \(b\)-th
one of \(x^{(h-1)}\). These intervals are disjoint and ordered.
If their left endpoints in row \(i\) are
\(l_{i,0}<\cdots<l_{i,k-1}\), extend them to the partition
\[
 [0,l_{i,1}-1],\ [l_{i,1},l_{i,2}-1],\ \ldots,\
 [l_{i,k-1},w-1].
\]
Each part is nonempty and contains its original interval. Horizontal
path edges connect the cells within each part and also join consecutive
parts. The link from the \(b\)-th one of \(y^{(i)}\) to the
\(b\)-th one of \(z^{(i)}=x^{(i+1)}\) supplies the vertical grid
edge with that label. Its coefficient is one by
Lemma~\ref{lem:support}. These parts and links therefore give the
full \(h\times k\) minor, including when an outgoing column is an
outer column. The partial row lies outside the branch sets.

The row computations cost \(O(hw\log w)=O(N\log N)\) bit
operations. Each row of \(I+R\) has at most three ones, and each
row of \(B\) has at most three. Using length-\(N\) bit rows,
the two products defining \(A\) therefore use \(O(N^2)\) bit
operations and \(O(N^2)\) bits of space. Listing the support, rules and minor witnesses
fits within the same bounds.
\end{proof}

The geometric part of this proof applies to any sequence of legal
routes whose output words are the next input words. Only the estimate
of the sum in \eqref{eq:route-count} used the single-gap cost bound.
We now improve that estimate at even width.

\begin{lemma}
\label{lem:pair6}
Let \(w\ge6\) be even, \(x\in\mathcal L_w\), and
\(q^0,q^1\in\F^w\). There are legal routes
\((a^0,y^0,u)\) for \((x,q^0)\) and \((a^1,y^1,z)\) for
\((u,q^1)\), with the same intermediate word \(u\), such that
\begin{equation}
 c(a^0,q^0)+c(a^1,q^1)-4k+D(z)-D(x)\le6.
 \label{eq:pair6}
\end{equation}
A deterministic algorithm chooses both routes from the complete input
in \(O(w\log w)\) bit time and space.
\end{lemma}

Applying Lemma~\ref{lem:transport-row} twice would give eight in
place of six. Lemma~\ref{lem:pair6}, called PAIR6 in the
supplement, chooses the intermediate word together with both routes.
Neither \(c(a^0,q^0)-2k+D(u)-D(x)\) nor
\(c(a^1,q^1)-2k+D(z)-D(u)\) is required to be at most four.
The first route may use more when the second uses less, so their
columns are chosen together. The common word is retained exactly when the two routes are
assembled into the minor. The proof in Appendix~\ref{app:pair6}
synchronises two copies of the row dynamic program and verifies a
finite upper bound on their combined cost.

One more improvement is available at the beginning of the construction.
For even width, the standard word has the fixed form
\(s_w=0(10)^k0\).

\begin{lemma}
\label{lem:first3}
Let \(w\ge6\) be even and \(q\in\F^w\). There is a legal route
\((a,y,z)\) for \((s_w,q)\) satisfying
\begin{equation}
 c(a,q)-2k+D(z)\le3.
 \label{eq:first3}
\end{equation}
Such a route can be computed in \(O(w\log w)\) bit time and space.
\end{lemma}

This standard-start inequality, called FIRST3 in the supplement,
restricts only the input interface. Its output remains in the full
class \(\mathcal L_w\), so it can be followed by
Lemma~\ref{lem:pair6}. In Appendix~\ref{app:first3}, we group the
interior columns into their fixed \(10\) blocks and prove the
inequality by a finite certificate and induction on the number of
blocks. The original row optimiser then attains the improved bound.

\begin{proof}[Proof of Theorem~\ref{thm:main}]
Construct the path diagonal by Proposition~\ref{prop:seed}, place
it on \(\Gamma_N\), and set
\(q^{(i)}_j=d_{i,j}\mathbin\oplus d_{i+1,j}\) for every full-row
gap. Start with \(x^{(0)}=s_w\), and always pass the output interface
of a route to the next route. The assembly in
Proposition~\ref{prop:transport-lift} preserves similarity, locality,
the path, all three degree bounds and the full minor for every legal
route sequence. It does not depend on how the gaps are grouped.
In particular, all routing words are ordered by increasing physical
column, and the direction of each shear is determined by the parity
of its actual rows.

Put \(t=h-1\), write \(c_i=c(a^{(i)},q^{(i)})\), and define
\[
 \delta_i=c_i-2k+D(x^{(i+1)})-D(x^{(i)}).
\]
Since \(D(x^{(0)})=0\), the exact count becomes
\begin{equation}
 m(A)=2N-2+2kt+\sum_{i=0}^{t-1}\delta_i-D(x^{(t)}).
 \label{eq:strict-telescope}
\end{equation}
The canonical parameters satisfy \(w\le h\le w+2\). If
\(h=w+2\), the inequality \(N<(w+1)^2\) forces \(r=0\).
We use the following three cases.

If \(w\) is odd, apply Lemma~\ref{lem:transport-row} at every
gap. Since \(2k=w-3\) and \(\delta_i\le4\),
\[
 m(A)\le2N-2+(h-1)(w+1)
       =3N+h-w-r-3\le3N-r-1.
\]

If \(w\) is even and \(h\le w+1\), apply
Lemma~\ref{lem:pair6} to gaps \((0,1),(2,3),\ldots\), followed
by Lemma~\ref{lem:transport-row} if a final gap remains. Each pair
starts at the output word of the preceding pair. Then
\[
 \sum_i\delta_i\le6\lfloor t/2\rfloor+4(t\bmod2)
                  =3t+(t\bmod2).
\]
Using \(2k=w-2\) in \eqref{eq:strict-telescope} gives
\begin{equation}
 m(A)\le3N+h-w-r-3+((h-1)\bmod2)
       =3N-r-2.
 \label{eq:even-budget}
\end{equation}
The last equality holds for both \(h=w\) and \(h=w+1\).

It remains to consider even \(w\) with \(h=w+2\) and \(r=0\),
so \(N=w(w+2)\). There are \(w+1\) gaps. Apply
Lemma~\ref{lem:first3} to gap zero, then
Lemma~\ref{lem:pair6} to
\[
 (1,2),(3,4),\ldots,(w-1,w).
\]
The first route therefore has \(\delta_0\le3\), and each following
pair has combined \(\delta\) at most six. Consequently
\begin{align*}
 m(A)&\le2N-2+(w-2)(w+1)+3+3w-D(x^{(w+1)})\\
     &=3N-1-D(x^{(w+1)})\le3N-1.
\end{align*}
The saving can also be written as an exact nonnegative remainder:
\begin{equation}
 \begin{split}
 H'&=D(x^{(w+1)})+(3-\delta_0)
      +\sum_{b=0}^{w/2-1}(6-\delta_{2b+1}-\delta_{2b+2})\ge0,\\
 m(A)&=3N-1-H'.
 \end{split}
 \label{eq:strict-remainder}
\end{equation}
Placing the unpaired gap first is what makes the standard-start
estimate available. All subsequent estimates use the interfaces
actually returned by the preceding routes.

The geometric estimates in the proof of Theorem~\ref{thm:paired}
give aspect ratio at most \(4/3\) and the stated square minor.
The routing takes \(O(hw\log w)\) bit operations; the simultaneous
similarity and output take \(O(N^2)\) bit time and space as before.
Adding the seed synthesis gives \(O(N^3)\) bit time and
\(O(N^2)\) bit space. The construction and these bounds apply to
every supplied irreducible polynomial of degree \(N\).
\end{proof}

\subsection{A computer-assisted refinement for finite widths}

For widths six through eleven, three specialised finite tables improve
on the row-cost bound in Proposition~\ref{prop:transport-lift}. They use the
same support formulas, choosing a mask from the diagonal differences
and the incoming branch-set columns. The resulting total upper bound
is smaller than that in Proposition~\ref{prop:transport-lift} by
\(2(h-1)\). We state the properties verified for
these tables, then derive the construction for every height.

Put \(W=2\lfloor w/2\rfloor\) and \(k=(W-2)/2\).
A \emph{routing state} is an increasing tuple
\(x=(x_0,\ldots,x_{k-1})\) of columns in \(\{0,\ldots,W-1\}\).
Its entries are the designated incoming columns for the branch sets in
one complete row. A table row selects outgoing columns in that row and
incoming columns in the next row. Here incoming and outgoing refer to
the order of the rows in the grid-minor construction; the links used
for this routing are undirected. The designated initial states and the
sizes of the state classes
\(\mathcal S_W\) are given in Table~\ref{tab:policies}.
For \(W=6\), the class contains all two-subsets; for \(W=8\),
it contains all three-subsets except \((0,1,2)\) and \((5,6,7)\).
For \(W=10\), it is the explicit list of 128 four-subsets specified
by the \texttt{class\_indices} field of \texttt{TABLES.json}, using
lexicographic indices starting at zero. This file is part of the
computational certificate described in Appendix~\ref{app:verification}.

Use potential zero for \(W=6,10\). For \(W=8\), put
\[
 \Phi_8(x)=\min\left\{3,
  \sum_{b=0}^{2}\operatorname{dist}(x_b,[3b,3b+1])
  +\one_{\{x\in\{(2,3,5),(2,4,5)\}\}}\right\},
\]
where \(\operatorname{dist}\) is ordinary distance on the real line.
Write these potentials uniformly as \(\Phi_W\).
They are nonnegative and vanish at the designated initial states.

For a mask \(a\in\F^W\) with \(a_0=a_{W-1}=0\) and an
internal difference word \(q\in\F^{W-2}\), extend \(q\) by zero
at both ends and define, as ordinary integers,
\[
 \tau_j=a_j\mathbin\oplus a_{j+1}\quad(0\le j<W-1),\qquad
 v_j=\tau_{j-1}+\tau_j+a_jq_j,
 \qquad
 c(a,q)=2\sum_{j=0}^{W-2}\tau_j+\sum_{j=0}^{W-1}a_jq_j,
\]
with \(\tau_{-1}=\tau_{W-1}=0\). Here \(v_j\) counts the new
neighbours contributed by this gap to a cell in column \(j\), and
\(c(a,q)\) counts the directed dependencies added across the gap.
An available link \((s,t)\) joins column \(s\) in the current row to
column \(t\) in the next row. The available links are \((j,j+1)\)
and \((j+1,j)\) when \(\tau_j=1\), and \((j,j)\) when
\(a_jq_j=1\). These are the undirected links arising from the
coefficient formulas in Lemma~\ref{lem:support}.

\begin{table}[htbp]
\centering
\begin{tabular}{@{}rrrrrc@{}}
\toprule
\(W\) & \(k\) & States & Words per state & Table rows & Initial state\\
\midrule
6 & 2 & 15 & 16 & 240 & \((0,3)\)\\
8 & 3 & 54 & 64 & 3456 & \((0,3,6)\)\\
10 & 4 & 128 & 256 & 32768 & \((0,3,6,9)\)\\
\bottomrule
\end{tabular}
\caption{The complete finite policies used for the widths
\(6\le w\le11\). Each table contains one row for every pair of
an admissible state and an internal difference word.}
\label{tab:policies}
\end{table}

\begin{lemma}
\label{lem:policies}
For every \(W\in\{6,8,10\}\), \(x\in\mathcal S_W\) and
internal word \(q\), the corresponding table row specifies a mask
\(a\), an increasing \(k\)-tuple \(y\) of columns in
\(\{0,\ldots,W-1\}\), and a successor
\(z\in\mathcal S_W\) such that
\[
 \max(x_b,y_b)<\min(x_{b+1},y_{b+1})\quad(0\le b<k-1),
 \qquad v_j\le2\quad(0\le j<W),
\]
every \((y_b,z_b)\) is an available link, and
\begin{equation}
 c(a,q)+\Phi_W(z)\le W+\Phi_W(x).
 \label{eq:policy}
\end{equation}
\end{lemma}

\begin{proof}[Computer-assisted verification]
The supplied verifier checks all 36,464 table rows using exact integer
arithmetic. It checks that each table has exactly one entry for every pair of an
admissible routing state and an internal difference word, with no other
entries, and that
each successor lies in the stated class. For each row it reconstructs
the mask, available links, incidences and cost, and verifies the displayed
ordering, incidence and potential inequalities. The file hash, command
and recorded results appear in Appendix~\ref{app:verification}.
This exhaustive verification proves the finite assertion.
\end{proof}

\begin{theorem}
\label{thm:finitewidth}
Let \(h\ge2\), \(6\le w\le11\), \(0\le r<w\), and
\(N=hw+r\). For every binary word \(d\in\F^N\), a deterministic
algorithm using the three fixed tables constructs a matrix \(A\) on
\(\Gamma_{h,w,r}\) similar to \(T(d)\). It is Moore-local with a
null boundary, retains the bidirectional spanning path, and has directed
and underlying maximum degrees at most six. Its underlying graph has
an explicit \(h\times\lfloor(w-2)/2\rfloor\) grid minor, and
\begin{equation}
 m(A)\le2N-2+W(h-1)
 =3N-\bigl[h(w-W)+r+W+2\bigr]<3N.
 \label{eq:finitewidth}
\end{equation}
Combining this construction with Proposition~\ref{prop:seed} uses
\(O(N^3)\) bit operations and \(O(N^2)\) bits of space for a
supplied irreducible polynomial. The proof depends on the finite
verification in Lemma~\ref{lem:policies}.
\end{theorem}

\begin{proof}
Place \(T(d)\) in serpentine order and write the placed matrix as
\(B=H+V+D\). Start in the designated state \(x^0\).
For gap \(g\), read the internal bits of
\(q^g_j=d_{g,j}\mathbin\oplus d_{g+1,j}\), and use the table row
for \((x^g,q^g)\) to obtain the mask \(a^g\), outgoing columns
\(y^g\), and next state \(x^{g+1}\).
If \(w=W+1\), append a zero mask bit at column \(W\).
The mask is zero at both the right boundary of the width-\(W\) strip
and the right boundary of the physical region, so no extra transition
or cost appears outside the strip.

In each gap, put a shear at every selected column, directed from the
odd complete row to the even complete row. The masks need not be
complementary, and these shears need not form a matching. Nevertheless,
row parity gives \(R^2=RHR=RDR=0\), and their internal support gives
\(RV=VR=0\). Thus \(A=(I+R)B(I+R)\) is similar to \(B\).
Lemma~\ref{lem:support} gives its actual added entries and shows that
the path is retained and every dependency is local. The partial row
and its turn remain unchanged.

Across gap \(g\), a cell in column \(j\) has \(v_j\le2\) new
neighbours. At most two complete-row gaps meet a cell, and the path
contributes at most two neighbours. Hence the underlying maximum degree
is at most six, and so are both directed maximum degrees.
The exact count is
\[
 m(A)=2N-2+\sum_{g=0}^{h-2}c(a^g,q^g).
\]
Summing \eqref{eq:policy} cancels the intermediate potentials.
Since \(\Phi_W(x^0)=0\) and \(\Phi_W\ge0\), it gives
\(\sum_g c(a^g,q^g)\le W(h-1)\), proving
\eqref{eq:finitewidth}.

It remains to specify the full minor. For each \(i<h-1\), the intervals
\[
 [\min(x^i_b,y^i_b),\max(x^i_b,y^i_b)]\quad(0\le b<k)
\]
are disjoint and ordered by Lemma~\ref{lem:policies}.
For the last row, set \(y^{h-1}=x^{h-1}\).
Writing \(l_b=\min(x^i_b,y^i_b)\) in each row, extend these intervals to
\[
 [0,l_1-1],\ [l_1,l_2-1],\ \ldots,\ [l_{k-1},w-1].
\]
They partition the row into \(k\) nonempty connected branch sets,
each containing its designated incoming and outgoing columns.
Horizontal path edges give the internal trees and the adjacencies
between consecutive branch sets. The available link from \(y^i_b\)
to \(x^{i+1}_b\) gives every vertical grid adjacency.
By the support calculation, these are actual nonzero dependencies
after cancellation. An appended physical column belongs to the last
branch set, while the partial row remains outside the minor.
The algorithm outputs every branch vertex, internal tree edge and
connecting edge.

The three tables have fixed finite size. The algorithm reads one table
entry for each gap and lists the matrix entries, local rules and
grid-minor witnesses. These lists contain \(O(N)\) items with
\(O(\log N)\)-bit indices. Sorting costs
\(O(N\log^2N)\) bit operations and the additional storage is
\(O(N\log N)\) bits. Together with Proposition~\ref{prop:seed},
these costs give the stated cubic time and quadratic space bounds.
\end{proof}

\begin{corollary}
\label{cor:finitewidth}
For every \(36\le N\le143\) and every supplied monic irreducible
\(p\in\F[x]\) of degree \(N\), the conclusions of
Theorem~\ref{thm:main} hold with the stronger bound
\(m(A)\le2N-2+2\lfloor w/2\rfloor(h-1)\), where \(h,w\)
are the canonical parameters.
If \(p\) is primitive, the resulting automaton has one cycle through
all \(2^N-1\) nonzero states.
\end{corollary}

\begin{proof}
The canonical width \(\lfloor\sqrt N\rfloor\) lies between six
and eleven exactly in this range. Apply Theorem~\ref{thm:finitewidth}
to the diagonal word from Proposition~\ref{prop:seed}. The period
conclusion follows from the same field argument as
Corollary~\ref{cor:primitive}.
\end{proof}

These specialised tables give the smaller bound
\(2N-2+W(h-1)\) on the stated widths. The transport construction in
Section~\ref{sec:transport} applies to every width and establishes the
uniform strict bound \(m(A)\le3N-1\). The optimal leading
coefficient under the prescribed-polynomial constraint remains
undetermined, between the lower coefficient \(5/2\) of
Theorem~\ref{thm:lower} and the upper coefficient \(3\) of
Theorem~\ref{thm:main}.

\section{Maximum-period constructions}
\label{sec:period}

\subsection{A supplied primitive polynomial}

\begin{proof}[Proof of Corollary~\ref{cor:primitive}]
The minimal polynomial of \(A\) is a nonconstant monic divisor
of its irreducible characteristic polynomial \(p\), and hence
equals \(p\). The field
\[
 K=\F[x]/(p)\cong\F[A]
\]
therefore acts on \(\F^N\), which is a one-dimensional vector
space over \(K\). Write \(\alpha=x\bmod p\).
For a nonzero state \(v\) and an integer \(j>0\),
\[
 A^jv=v
 \quad\Longleftrightarrow\quad
 (\alpha^j-1)v=0
 \quad\Longleftrightarrow\quad
 \alpha^j=1.
\]
Primitivity gives \(\ord(\alpha)=2^N-1\), so every nonzero
state has that period. There are exactly \(2^N-1\) such states,
and they consequently form one cycle. The zero state is fixed
because the update is linear.
\end{proof}

\subsection{A supplied factorisation}

We give the arithmetic reduction explicitly to identify the input
required by Corollary~\ref{cor:factors}. Once the factor data have
been verified, the following reduction applies to every degree \(n\ge2\).

\begin{proposition}
\label{prop:primitive-acquisition}
Given \(n\ge2\) and a complete prime factorisation
\begin{equation}
  M=2^n-1=\prod_{j=1}^{\rho}q_j^{e_j},
  \label{eq:factorisation}
\end{equation}
where the \(q_j\) are distinct primes and \(e_j>0\), a
deterministic algorithm constructs a monic primitive polynomial
of degree \(n\) over \(\F\) in \(n^{O(1)}\) bit operations.
\end{proposition}

\begin{proof}
Shoup's irreducible-polynomial algorithm~\cite[Theorem~3.2]{S90}
constructs a monic irreducible \(f\in\F[x]\) of degree \(n\)
using \(O(n^4 P(\log n))\) base-field operations for a fixed polynomial
\(P\), and hence \(O(n^5)\) bit operations over \(\F\).
Here we use the source's convention that its \(n^\epsilon\) factors
denote fixed polynomials in \(\log n\). We work in the polynomial-basis
representation \(K=\F[x]/(f)\).

We also use Shoup's primitive-element
theorem~\cite[Theorem~1]{S92}. Specialised to \(\F\), it states
that an absolute constant \(c>0\) exists such that, whenever
\begin{equation}
 2^{\lambda}>c\,\rho^4(1+\log \rho)^4n^2,
 \label{eq:shoup}
\end{equation}
some monic irreducible \(\theta\in\F[x]\) of degree \(\lambda\)
has primitive residue modulo \(f\).
This fixed-characteristic result is unconditional.
Since
\[
 2^\rho\le\prod_{j=1}^\rho q_j\le 2^n-1<2^n,
\]
we have \(\rho<n\). Increase \(c\) to at least one, set
\[
 B=c\,\rho^4(1+\log_2\rho)^4n^2,
\]
and let \(\lambda_*\) be the least positive integer with
\(2^{\lambda_*}>B\). Since \(B\ge4\),
\[
 2^{\lambda_*}\le2B.
\]
Enumerate the monic binary polynomials of degrees \(0,\ldots,n-1\)
in degree order, using a fixed coefficient-vector order within each
degree. They represent every nonzero element of \(K\) once.
For each candidate, let \(a\) be its residue modulo \(f\),
and test
\begin{equation}
 a^{M/q_j}\ne1
 \qquad(1\le j\le\rho).
 \label{eq:ordertest}
\end{equation}
Stop at the first successful candidate. Every nonzero element of
\(K\) has order \(d\mid M\). If \(d<M\), some prime \(q_j\)
divides \(M/d\), giving \(d\mid M/q_j\); conversely, an element
of order \(M\) passes every test. Hence \eqref{eq:ordertest} is
equivalent to primitivity.
If \(\lambda_*<n\), the candidate guaranteed by
\eqref{eq:shoup} occurs by that degree, and fewer than
\(2^{\lambda_*+1}\le4B\) candidates are tested.
If \(\lambda_*\ge n\), a primitive element occurs somewhere in
the full enumeration, because \(K^\times\) is cyclic; in this
case \(2^n\le2^{\lambda_*}\le2B\).
Thus the same bound of \(4B\) candidates covers both cases,
including the small degrees. The algorithm need not know \(c\)
or \(\lambda_*\): it simply stops when
\eqref{eq:ordertest} holds. The bound controls that stopping
time. Candidates need not themselves be irreducible.

Using schoolbook arithmetic, multiplication in \(K\) costs
\(O(n^2)\) bit operations. Each exponent in
\eqref{eq:ordertest} has at most \(n\) bits, so a modular
exponentiation costs \(O(n^3)\). Since \(\rho<n\), each candidate
requires \(O(n^4)\) bit operations.
Since \(1+\log_2\rho\le n\) for \(n\ge2\), we have \(B\le c n^{10}\).
The search up to the first successful candidate therefore costs
\(O(n^{14})\) bit operations.

Finally, the selected element \(a\) has degree \(n\) over
\(\F\). Otherwise \(\F(a)\) would have order \(2^m\) for some
\(m<n\), forcing \(\ord(a)\le2^m-1<M\).
Compute \(1,a,\ldots,a^n\) and solve
\[
 a^n=\sum_{i=0}^{n-1}b_i a^i
\]
by binary Gaussian elimination. The polynomial
\(x^n+\sum_{i=0}^{n-1}b_i x^i\) is the minimal polynomial
of \(a\), and is primitive. Computing the powers and solving the system cost \(O(n^3)\) bit
operations. This completes the running-time estimate.
\end{proof}

\begin{proof}[Proof of Corollary~\ref{cor:factors}]
Verify the factor data by Appendix~\ref{app:certificates} and
apply Proposition~\ref{prop:primitive-acquisition} with \(n=N\)
to obtain a primitive polynomial. Theorem~\ref{thm:main} then
synthesises the automaton in cubic time, retaining all the stated
geometric properties. Corollary~\ref{cor:primitive} gives the single
cycle through its nonzero states.
Reading and verifying the factor data costs \(O((N+L)^4)\);
constructing the field model, selecting a primitive element
and synthesising the automaton together cost \(O(N^{14})\)
bit operations.
These give \(O((N+L)^{14})\) elementary bit operations.
Charging an additional logarithmic factor for explicit indices
and addressing still gives \(K(N+L)^{16}\) for a fixed \(K\).
\end{proof}

The complete list of prime divisors in \eqref{eq:factorisation}
is essential to the order test. A test using only some of the
prime divisors may accept an element of smaller order.
The certificate check establishes this completeness before
the arithmetic search begins. The exponent sixteen is a conservative
bound for the combined algorithm, rather than a claim of practical
running time.

\section{Examples}
\label{sec:examples}

The three primitive inputs
\[
 \begin{aligned}
 p_{36}(x)&=x^{36}+x^{11}+1,\\
 p_{37}(x)&=x^{37}+x^5+x^4+x^3+x^2+x+1,\\
 p_{64}(x)&=x^{64}+x^4+x^3+x+1
 \end{aligned}
\]
give the examples in Table~\ref{tab:examples}.
These examples use the paired-shear construction of
Theorem~\ref{thm:paired}, with the three primitive polynomials supplied
as input.
The \(N=37\) example uses a \(6\times6\) square and one additional
cell, as in Figure~\ref{fig:fringe}. It illustrates the construction
at a prime degree with exactly \(37\) cells.

\begin{table}[htbp]
\centering
\begin{tabular}{@{}rrrrcc@{}}
\toprule
Region & \(N\) & \(m(A)\) & \(\ell(A)\) & Grid minor &
Maximum in/out degree\\
\midrule
\(6\times6\) & 36 & 102 & 14 & \(6\times2\) & \(6/6\)\\
\(6\times6+1\) cell & 37 & 103 & 16 & \(6\times2\) & \(6/6\)\\
\(8\times8\) & 64 & 195 & 38 & \(8\times3\) & \(6/6\)\\
\bottomrule
\end{tabular}
\caption{Three exact examples obtained by constructing the tridiagonal
path matrix and applying the paired shear. Degrees exclude self-dependence.}
\label{tab:examples}
\end{table}

For the same three inputs, the finite-width construction in
Theorem~\ref{thm:finitewidth} gives directed connection counts
\(94\), \(95\) and \(168\), respectively. Its separate data include
the masks, routing states and full minor certificates. The uniform
construction of Theorem~\ref{thm:main} gives counts
\(94\), \(96\) and \(166\),
respectively, on these inputs. The smaller worst-case bound of the specialised tables does not imply
a smaller connection count for each individual input.

For each example, the accompanying data specify the matrix entries,
the nine-bit local rules and the complete minor witness. The
input polynomial is checked for irreducibility and primitivity,
using complete integer factorisations and exact
modular arithmetic. A separate verifier computes the
characteristic polynomial from the full matrix and checks
locality, strong connectivity, rule encoding and every minor
witness edge. These checks illustrate the construction and its
implementation; the general theorem follows from the preceding
proofs. Appendix~\ref{app:verification} gives the precise
verification procedure and finite test ranges.

\appendix
\section{Deterministic synthesis of the path matrix}
\label{app:seed}

We prove Proposition~\ref{prop:seed}. All arithmetic in this
appendix is over \(\F\), and \(p\) is monic and irreducible of
degree \(n\ge2\). The linear system below is the construction of
Mesirov and Sweet~\cite{MS87}. We include the proof to specify a
deterministic implementation and its cost.

\subsection{The linear system}

Set \(u_j=x^j\bmod p\) for \(0\le j\le2n\), with coefficient
vectors in increasing order of degree. Form the
\((n+1)\times n\) matrix \(M\) with rows
\[
 M_0=(1,0,\ldots,0),\qquad
 M_i=\coeff(u_{i-1}+u_{2i-1}+u_{2i})
 \quad(1\le i\le n).
\]
Solve
\begin{equation}
 Mf=(1,0,\ldots,0,1)^{\mathsf T},
 \qquad f=(f_1,\ldots,f_n)^{\mathsf T},
 \label{eq:seed-system}
\end{equation}
with a fixed left-to-right pivot order, setting free variables
to zero. Define
\begin{equation}
 q(z)=\left[p(z)\sum_{i=1}^n f_i z^{-i}\right]_{\ge0},
 \label{eq:numerator}
\end{equation}
where the subscript retains the terms with nonnegative powers.
We will prove that the Euclidean algorithm for \((p,q)\)
has precisely \(n\) linear quotients, including its final
division by one.

\begin{lemma}
\label{lem:rank}
The system \eqref{eq:seed-system} is consistent and has exactly
two solutions.
\end{lemma}

\begin{proof}
Let \(K=\F[x]/(p)\), \(\alpha=x\bmod p\), and
\(a=\alpha+\alpha^2\). Irreducibility and \(n\ge2\) imply
\(\alpha\ne0,1\), so \(a\ne0\).
Write \(M_*\) for the last \(n\) rows of \(M\).
For a row vector \(v=(v_1,\ldots,v_n)\), put
\(\rho=\sum_{i=1}^n v_i\alpha^{i-1}\).
Then \(vM_*\) is the coefficient vector of \(\rho+a\rho^2\).
This map is \(\F\)-linear, and
\[
 \rho+a\rho^2=0
 \quad\Longleftrightarrow\quad
 \rho=0\ \text{or}\ \rho=a^{-1}.
\]
Hence \(\rank M_*=n-1\).

Moreover,
\[
 a^{-1}=\alpha^{-1}+(\alpha+1)^{-1}.
\]
Since \(p(0)=p(1)=1\), the representatives of these last two
inverses are \((p(x)+1)/x\) and \((p(x)+1)/(x+1)\).
Both have degree \(n-1\) and leading coefficient one.
Their sum has zero coefficient at \(x^{n-1}\).
The left kernel of \(M_*\) is therefore orthogonal to
\(e_n=(0,\ldots,0,1)^{\mathsf T}\).
It follows that \(M_*f=e_n\) is consistent and has two
solutions.

Now choose \(v\) representing \(\rho=\alpha^{-1}\).
We have \(\rho+a\rho^2=1\) and \(v_n=1\), so
\(vM_*=M_0\). Every solution of \(M_*f=e_n\) satisfies
\(M_0f=ve_n=1\). Adjoining the first row of \(M\) thus changes
neither the rank nor the solution set.
\end{proof}

\subsection{A nondegenerate sequence of Hankel forms}

Define the linear functional \(L:K\to\F\) by
\(L(\alpha^{i-1})=f_i\) for \(1\le i\le n\).
The system gives \(L(1)=1\) and
\[
 L\bigl(\alpha^j+(\alpha+\alpha^2)\alpha^{2j}\bigr)=0
 \qquad(0\le j\le n-2).
\]
By linearity of squaring over \(\F\), every polynomial \(u\)
of degree at most \(n-2\) satisfies
\begin{equation}
 L\bigl(u(\alpha)+(\alpha+\alpha^2)u(\alpha)^2\bigr)=0.
 \label{eq:moment}
\end{equation}

\begin{lemma}
\label{lem:hankel}
For every \(1\le m\le n\), the Hankel matrix
\[
 H_m=\bigl(L(\alpha^{i+j})\bigr)_{0\le i,j<m}
\]
is nonsingular.
\end{lemma}

\begin{proof}
Suppose that \(H_m\) is singular. Identify its kernel vectors with
polynomials of degree less than \(m\), and choose a nonzero such
polynomial \(v\) of least degree. Thus
\[
 L(v(\alpha)u(\alpha))=0\qquad(\deg u<m).
\]
The polynomial \(v\) is not constant because \(L(1)=1\).
Write \(v(x)=b+xw(x)\), where \(b\in\F\) and
\(\deg w\le m-2\). Suppress evaluation at \(\alpha\) inside
\(L\). The kernel relations with \(u=v,w\), together with
\eqref{eq:moment}, yield
\[
 0=L(v^2)=b+L(\alpha^2w^2)=b+L(w)+L(\alpha w^2),
 \qquad
 0=L(vw)=bL(w)+L(\alpha w^2).
\]
Adding gives \(b+(1+b)L(w)=0\). The case \(b=1\) is
impossible, so \(b=0\) and \(L(w)=0\).
For \(1\le j<m\),
\[
 L(\alpha^j w)=L(\alpha^{j-1}v)=0.
\]
Hence the nonzero polynomial \(w\), whose degree is smaller
than that of \(v\), lies in the same kernel, a contradiction.
\end{proof}

\subsection{The tridiagonal multiplication matrix}

Equip \(K\) with the symmetric bilinear form
\(\langle u,v\rangle=L(uv)\).
Apply Gram--Schmidt to \(1,\alpha,\ldots,\alpha^{n-1}\), in this order.
Lemma~\ref{lem:hankel} shows that the restriction of the form to
each initial polynomial subspace is nondegenerate.
At each step, let \(v\) be the vector remaining after subtracting the
earlier projections. Then \(\langle v,v\rangle\ne0\); otherwise \(v\)
would lie in the radical of the corresponding initial subspace.
Since the base field is \(\F\), we have \(\langle v,v\rangle=1\). We obtain monic polynomial
representatives \(P_0,\ldots,P_{n-1}\) with
\[
 P_0=1,\qquad \deg P_j=j,\qquad
 \langle P_i,P_j\rangle=\delta_{ij}.
\]

Multiplication by \(\alpha\) is self-adjoint for this form.
For \(i\le j-2\), the element \(\alpha P_i\) has a polynomial
representative of degree at most \(j-1\), giving
\(\langle P_i,\alpha P_j\rangle
=\langle\alpha P_i,P_j\rangle=0\).
For \(j<n-1\), the coefficient of \(P_{j+1}\)
in \(\alpha P_j\) is one by monicity.
Self-adjointness gives the same entry on
the other adjacent diagonal. Thus its matrix \(J\) in this
basis is symmetric tridiagonal, with ones on both adjacent
diagonals. In the polynomial basis, the same linear map has
the companion matrix of \(p\). Therefore \(\chi_J=p\).

\subsection{Euclidean quotients and the output word}

Extend \(f_i=L(\alpha^{i-1})\) to all \(i\ge1\).
Writing \(p(x)=x^n+\sum_{j=0}^{n-1}p_jx^j\) gives
\[
 f_{n+i}=\sum_{j=0}^{n-1}p_j f_{i+j}\qquad(i\ge1).
\]
The formal Laurent series \(F(z)=\sum_{i\ge1}f_i z^{-i}\)
then satisfies \(p(z)F(z)=q(z)\), with \(q\) as in
\eqref{eq:numerator}. Since \(f_1=1\), the polynomial \(q\)
is monic of degree \(n-1\).
In the orthonormal basis \(P_0,\ldots,P_{n-1}\), the coefficient of
\(P_0=1\) in \(u\) is \(L(u)\). Let \(e_0\) be the first standard
basis vector. Expanding the resolvent therefore gives
\[
 F(z)=e_0^{\mathsf T}(zI-J)^{-1}e_0.
\]
Cramer's rule identifies \(q\) as the characteristic
polynomial of the trailing \((n-1)\times(n-1)\) submatrix
of \(J\).

Write the diagonal of \(J\) as \(a_1,\ldots,a_n\), using
one-based indices in this paragraph. Put
\[
 Q_j(z)=\det\bigl(zI-J[\{j+1,\ldots,n\}]\bigr)
 \quad(0\le j\le n),\qquad Q_{n+1}=0.
\]
Thus \(Q_0=p\), \(Q_1=q\), and \(Q_n=1\).
Expansion of the tridiagonal determinant gives
\[
 Q_j=(z+a_{j+1})Q_{j+1}+Q_{j+2}
 \qquad(0\le j<n).
\]
For \(0\le j\le n\), the polynomial \(Q_j\) is monic of degree \(n-j\).
These equations are precisely the steps of the Euclidean
algorithm for \((p,q)\), including the division by \(Q_n=1\).
It follows that its quotients are
\(z+a_1,\ldots,z+a_n\).
Returning the word
\[
 d=(a_n,a_{n-1},\ldots,a_1)
\]
gives the reversal of the same path matrix, and hence
\(\chi_{T(d)}=p\). This procedure computes \(d\) without
constructing the orthogonal basis.

Computing \(x^j\bmod p\) successively for \(0\le j\le2n\)
takes \(O(n^2)\) bit operations and uses \(O(n^2)\) bits of space.
The binary system has \(n+1\) rows and \(n\) unknowns, so
Gaussian elimination takes \(O(n^3)\) bit operations and
\(O(n^2)\) bits of space. The numerator and the Euclidean
divisions also fit within this bound. This proves
Proposition~\ref{prop:seed}.

\section{Verifying the supplied factor data}
\label{app:certificates}

We use a recursive certificate based on the complete factorisation of
an integer's predecessor, following the method of Pratt~\cite{P75}.
The format below is chosen to make the input cost in
Corollary~\ref{cor:factors} explicit.

\subsection{Encoding and verification}

Fix the degree \(n\ge2\) and put \(M=2^n-1\).
All integers are encoded in canonical binary form, and the records and
child lists have explicit length delimiters. Certificates are trees:
every child is written in full, with no references to other nodes.
The total length \(L\) includes the top-level factor list and every
certificate tree. A parser checks all delimiters and consumes the
entire input.

The label \(2\) has a leaf certificate. A certificate for a label
\(m>2\) gives a witness \(2\le a<m\), pairwise distinct child labels
\(s_j\), positive exponents \(b_j\), and a certificate for every child.
The verifier requires \(m\) to be odd, all node labels to lie in
\([2,M]\), every child label to be smaller than its parent, and
all exponents to lie in \([1,n]\). It verifies the children recursively,
checks
\begin{equation}
 m-1=\prod_j s_j^{b_j},
 \label{eq:cert-product}
\end{equation}
and tests
\begin{equation}
 a^{m-1}\equiv1\pmod m,\qquad
 \gcd\bigl(a^{(m-1)/s_j}-1,m\bigr)=1
 \quad\text{for every }j.
 \label{eq:cert-tests}
\end{equation}
In the greatest common divisor tests, the powers are first reduced
modulo \(m\).

The top-level list gives distinct labels \(q_j\), positive exponents
\(e_j\le n\), and a certificate for each \(q_j\).
After validating these records, the verifier checks that
\(\prod_jq_j^{e_j}=M\). This establishes the completeness needed
in \eqref{eq:ordertest}.
Malformed records, noncanonical integers, invalid labels and
unconsumed trailing data are rejected.

\subsection{Correctness and cost}

We prove by induction on the certificate tree that every accepted
label is prime. The leaf label \(2\) is prime. Suppose a nonleaf
label \(m\) passes, and let \(p_0\) be any prime divisor of \(m\).
The order \(d\) of \(a\) modulo \(p_0\) divides both \(m-1\)
and \(p_0-1\). For every \(s_j^{b_j}\) in
\eqref{eq:cert-product}, the greatest common divisor test implies
\(d\nmid(m-1)/s_j\). Since the children are prime, this forces
\(s_j^{b_j}\mid d\) for each \(j\), and hence \(d=m-1\).
Thus \(m-1\mid p_0-1\), giving \(p_0\ge m\).
As \(p_0\mid m\), we obtain \(p_0=m\), proving primality.

Conversely, every prime label has such a certificate.
For a prime \(m>2\), factor \(m-1\) completely, choose a generator
of \((\mathbb Z/m\mathbb Z)^\times\), and recursively certify the
prime factors of \(m-1\). The generator satisfies
\eqref{eq:cert-tests}. This proves existence of the input
certificates; their production is not a step of the synthesis algorithm.
The cyclicity used here, and for \(K^\times\) in
Proposition~\ref{prop:primitive-acquisition}, follows from the
root bound for polynomials over a field. For a finite field \(F\)
and each prime power \(\pi^b\) dividing \(|F^\times|\),
some element has order divisible
by \(\pi^b\); otherwise every element would be a root of
\(x^{|F^\times|/\pi}-1\). Taking elements of these prime-power
orders and multiplying them gives a generator.

There are \(O(L)\) nodes and child records.
Before any integer power is formed, the verifier checks its exponent
against \(n\). Product calculations are capped at their target:
a multiplication exceeding \(m-1\), or \(M\) at the top level,
causes rejection. Consequently, the retained operands have \(O(n)\)
bits, including for invalid input. Schoolbook modular exponentiation
and greatest common divisor calculations take \(O(n^3)\) bit
operations per child record. Pairwise distinctness tests cost at
most \(O(L^2n)\). Reading, parsing and all verification therefore
use \(O((n+L)^4)\) bit operations.
This bound also covers rejection; an incomplete list cannot reach
the primitive-element search.

\section{Exact degree formulas}
\label{app:degrees}

The following formulas give the individual degrees in
Proposition~\ref{prop:counts}, including boundary cells and
the possible partial row.
Let \(b(i,j)=1\) at the two endpoints of the serpentine path
on \(\Gamma_{h,w,r}\), and \(b(i,j)=2\) elsewhere.
Define
\[
 \alpha_j=a^+_j,\qquad \beta_j=a^-_j,
\]
with both masks zero outside \(1,\ldots,2k\). Put
\[
 \tau_\alpha(j)=\alpha_j\mathbin\oplus\alpha_{j+1},\qquad
 \tau_\beta(j)=\beta_j\mathbin\oplus\beta_{j+1}
 \quad(0\le j<w-1),
\]
and set both transition indicators to zero outside this range. Define
\[
 F_\alpha(j)=\tau_\alpha(j-1)+\tau_\alpha(j),\qquad
 F_\beta(j)=\tau_\beta(j-1)+\tau_\beta(j).
\]
These ordinary integer sums count the cross-column entries in
\eqref{eq:four} incident with a cell at column \(j\).

For an odd complete row \(i\), let
\begin{align*}
 \delta_o(i,j)&=F_\alpha(j)+\one_{\{i+1<h\}}F_\beta(j),\\
 \eta_o(i,j)&=\alpha_j(d_{i,j}\mathbin{\oplus}d_{i-1,j})
  +\one_{\{i+1<h\}}\beta_j(d_{i,j}\mathbin{\oplus}d_{i+1,j}).
\end{align*}
Then
\[
 \degin(i,j)=b(i,j),\qquad
 \degout(i,j)=b(i,j)+\delta_o(i,j)+\eta_o(i,j).
\]
For an even complete row \(i\), let
\begin{align*}
 \delta_e(i,j)&=\one_{\{i+1<h\}}F_\alpha(j)+\one_{\{i>0\}}F_\beta(j),\\
 \eta_e(i,j)&=\one_{\{i+1<h\}}\alpha_j(d_{i+1,j}\mathbin{\oplus}d_{i,j})
  +\one_{\{i>0\}}\beta_j(d_{i-1,j}\mathbin{\oplus}d_{i,j}).
\end{align*}
Then
\[
 \degout(i,j)=b(i,j),\qquad
 \degin(i,j)=b(i,j)+\delta_e(i,j)+\eta_e(i,j).
\]
Every term whose indicator is zero is omitted before any
diagonal entry is read. In particular, these formulas never
refer to an exterior cell. Every partial-row vertex has
\(\degin=\degout=b\), since no shear meets that row. We have
\(\delta_o,\delta_e\le3\) and \(\eta_o,\eta_e\le1\).
Counting self-dependence increases either directed degree bound
from six to seven. These formulas concern the paired shear;
for the transport and finite-width constructions, incidences are counted
gap by gap as in Proposition~\ref{prop:transport-lift} and
Theorem~\ref{thm:finitewidth}.

\section{Finite-state proofs of the routing inequalities}
\label{app:transport}

We first prove Lemma~\ref{lem:transport-row} using a dynamic program with a
fixed set of states, then synchronise two copies for
Lemma~\ref{lem:pair6} and restrict the input to the standard word
for Lemma~\ref{lem:first3}. In each case a finite certificate bounds the minimum accepting
cost for every input allowed by the lemma. The state transitions describe
the interval and link conditions exactly. The certificate is checked by integer arithmetic;
its closure under each input letter then gives the result for every
width by induction.

\subsection{States, transitions and weights}

Immediately before column \(j\), let \(X,Y,Z\) count the ones
already read in \(x,y,z\), respectively. Use the state
\[
 s=(\rho,\sigma,\alpha,\beta)
   =(X-Y,Y-Z,a_{j-1},a_j),
\]
where \(\rho,\sigma\in\{-1,0,1\}\) and
\(\alpha,\beta\in\{0,1\}\). There are \(36\) states, ordered
lexicographically and numbered from zero. The initial and accepting
state is \(s_*=(0,0,0,0)\), with index \(16\). The virtual mask
bit before column zero is zero.

At a column, the input consists of its bits \(x_j,q_j\) and the
balance \(b=B_j(x)\). Since \(X-Z=\rho+\sigma\), the current
output balance is \(b_z=b-2(\rho+\sigma)\). Choose
\((\gamma,y_j,z_j)=(a_{j+1},y_j,z_j)\), and put
\[
 L=\alpha\mathbin\oplus\beta,\qquad
 R=\beta\mathbin\oplus\gamma,\qquad V=\beta q_j.
\]
Retain only choices with \(L+R+V\le2\). The successor differences
\(\rho',\sigma'\) are determined by the following rules.

For the horizontal intervals, use
\[
\begin{array}{c|c|c}
 \rho&\text{requirement}&\rho'\\ \hline
 0&\text{none}&x_j-y_j\\
 1&x_j=0&1-y_j\\
 -1&y_j=0&-1+x_j
\end{array}
\]
For the links to the next row, a nonzero \(\sigma\) first requires
\(L=1\). If \(\sigma=0\), then \(y_j=z_j=0\) gives
\(\sigma'=0\); \(y_j=z_j=1\) requires \(V=1\) and gives
\(\sigma'=0\); and unequal bits require \(R=1\) and give
\(\sigma'=y_j-z_j\). If \(\sigma=1\), require \(z_j=1\),
and also \(R=1\) when \(y_j=1\), and set \(\sigma'=y_j\).
If \(\sigma=-1\), require \(y_j=1\), and also \(R=1\) when
\(z_j=1\), and set \(\sigma'=-z_j\).

The input and output balances, both before and after the column, must
lie in \([-3,3]\). Their next values are
\(b+2x_j-1\) and \(b+2x_j-1-2(\rho'+\sigma')\).
The next state is \((\rho',\sigma',\beta,\gamma)\).
At the first and last columns impose \(z_j=0\); at the last column
also impose \(\beta=\gamma=0\). Begin and end at \(s_*\).

These rules are equivalent to the geometric conditions of
Lemma~\ref{lem:transport-row}. For the horizontal intervals joining the one positions of \(x\) and
\(y\), strict separation allows at most one unfinished interval. Its orientation is recorded by \(\rho\). If an \(x\)
endpoint has opened it, another \(x\) endpoint is forbidden until
after it closes; the symmetric rule applies to a \(y\) endpoint.
The prohibition also applies in the closing column, excluding a shared
endpoint. Simultaneous endpoints at \(\rho=0\) give a singleton.
This proves the three transition rules for \(\rho\) in both
directions.

For the links joining the one positions of \(y\) and \(z\), displacement
at most one means that a pending link must close at column \(j\) using
an available link between columns \(j-1\) and \(j\). If column \(j\)
also opens the next link, an available link between columns \(j\) and
\(j+1\) is required. With no pending link, simultaneous endpoints require vertical
support. These are precisely the rules for \(\sigma\), and they
preserve label order. Conversely, they reconstruct the ordered links
one by one. The incidence condition counts all added edges, whether or
not a chosen link uses them. Acceptance at \(s_*\) equates the
three numbers of ones and closes every pending interval and link.
Neither endpoint of \(y\) is forced to zero.

Assign each transition the integer weight
\begin{equation}
 \omega=2R+V-2x_j
       +|b-2(\rho+\sigma)+1|z_j-|b+1|x_j.
 \label{eq:transport-weight}
\end{equation}
The sum along an accepting path is exactly
\(c(a,q)-2k+D(z)-D(x)\), by \eqref{eq:transport-local}.
Each mask transition is counted at its left column, and the final
virtual transition is zero. The endpoint bits in both potentials are
zero. Thus minimum-cost dynamic programming computes the minimum of
this expression over all legal outputs. Negative weights cause no
difficulty: the input has fixed finite length and each transition
advances one column.

\subsection{A finite certificate for arbitrary width}

A cost vector has one coordinate in \(\mathbb Z\cup\{+\infty\}\)
for each of the \(36\) states. For an interior input letter
\((x,q)\), let \(T_{b,x,q}\) assign to each successor state the
minimum, over all permitted transitions into that state, of the
predecessor cost plus the transition weight.
Unreachable states have cost \(+\infty\). This update is monotone
in the componentwise order. Let \(\mathcal T_b\) append the actual
last zero column and then take the cost at \(s_*\); this terminal
operation is also monotone.

After the first zero column and some interior columns, record a context
\((b,\ell)\), where \(b\) is the input balance and \(\ell\)
is the number of processed interior columns, capped at four. The
initial context is \((-1,0)\). An interior letter \((x,q)\) is
allowed if \(-3\le b+2x-1\le3\); either value of \(q\) is
allowed. Its new context is
\((b+2x-1,\min\{4,\ell+1\})\).

The file \texttt{cover.json} in \texttt{supplement/transport/ft}
specifies a directed graph with \(277\) nodes and \(1072\) edges.
Each node has a context and a full \(36\)-coordinate vector \(V\).
Each edge is labelled by one of the four input letters
\((0,0),(0,1),(1,0),(1,1)\). Exact verification establishes:
\begin{enumerate}
 \item Every node is reachable from the designated root. The root has
       context \((-1,0)\), and its vector is a componentwise upper bound for
       the minimum-cost vector after the first column, whose finite entries are \(v_9=2\),
       \(v_{16}=0\), and \(v_{17}=2\).
 \item Each node has exactly one outgoing edge for each allowed letter,
       no edge for a forbidden letter, and the correct successor context.
 \item For each edge \(u\xrightarrow{(x,q)}v\),
       \begin{equation}
        V_v\ge T_{b_u,x,q}(V_u)
        \label{eq:transport-cover}
       \end{equation}
       holds in every coordinate, with \(+\infty\) larger than every
       integer.
 \item Every node with \(\ell=4\) and \(b\in\{-1,-2\}\) has
       finite terminal cost \(\mathcal T_b(V)\le4\). There are
       \(94\) such nodes. Both values of the last \(q\) bit give
       the same terminal cost because the last mask bit is zero.
\end{enumerate}
The checker reconstructs the transitions in two different encodings
and verifies these statements directly from the certificate entries.
It retains every vector coordinate, including infinite coordinates.
Reproduction details appear in Appendix~\ref{app:verification}.

To deduce the unrestricted result, take any input in
Lemma~\ref{lem:transport-row} and follow the certificate edges for its
interior letters. Its true minimum-cost vector is componentwise at
most the certificate vector. This holds at the root by the first check.
If it holds at a node, monotonicity and \eqref{eq:transport-cover}
give it at the next node. Since \(w\ge6\), the final interior
context has \(\ell=4\). The input has \(k\) ones and a last zero,
so its balance just before that zero is \(-1\) for even \(w\)
and \(-2\) for odd \(w\). The terminal check and monotonicity
therefore give a finite accepting cost at most four. By
\eqref{eq:transport-weight}, this is \eqref{eq:transport-row}.
The first and last \(q\) bits have no effect because the corresponding
mask bits vanish, so the argument covers every \(q\in\F^w\).

The certificate bounds the vector of minimum costs; it need not choose
an output before the whole input is known. To compute an output, run
the dynamic program on the supplied word, retain a minimising predecessor
at each state and column, and backtrack from the final state. Fixing an
enumeration order resolves ties deterministically. The number of states
and choices per column is bounded by an absolute constant, and every
transition weight lies
between \(-5\) and \(6\). Hence finite intermediate costs use
\(O(\log w)\) bits. The time and space bounds are
\(O(w\log w)\), completing the proof of
Lemma~\ref{lem:transport-row}.

\subsection{Joint routing across two gaps}
\label{app:pair6}

For the two routes in Lemma~\ref{lem:pair6}, keep a pair
\((s_0,s_1)\) of row states. There are \(36^2=1296\) states,
indexed by \(36s_0+s_1\); the initial and accepting state
\((s_*,s_*)\) has index \(592\). At a column with input balance
\(b=B_j(x)\), the first row transition reads \((x_j,q^0_j)\)
and outputs \(u_j\). If its current differences are
\(\rho_0,\sigma_0\), the second transition reads precisely
\[
 (u_j,q^1_j),\qquad b_1=b-2(\rho_0+\sigma_0).
\]
The second transition outputs \(z_j\). Both transitions use the
actual first- and last-column conditions above. Thus the middle bit
and its prefix balance are shared, and all three interface words
satisfy their balance constraints.

Every accepting path gives two legal routes by the single-row
correspondence. Acceptance equates the numbers of ones in
\(x,y^0,u,y^1,z\) and closes all pending intervals and links.
Conversely, two legal routes with the same middle word determine an
accepting path. In particular, the middle interface is not selected
independently by the two row computations.

Give a joint transition the sum of its two weights
\eqref{eq:transport-weight}. The contribution of \(u_j\) to the
two potentials cancels at each column. Along a complete path, the
cardinality terms sum to \(-2\#x-2\#u=-4k\), so the total
weight is exactly
\[
 c(a^0,q^0)+c(a^1,q^1)-4k+D(z)-D(x).
\]
This includes every support entry in both gaps, even when it is not
selected as a minor witness.

Let \(P_{b,x,q^0,q^1}\) be the resulting min-plus operator on
\(1296\)-coordinate cost vectors, and let \(\mathcal P_b\)
append the true last zero column and take the cost at state \(592\).
Both operators are monotone. The endpoint bits of \(q^0,q^1\)
have no effect because the endpoint masks are zero. The exact
post-first-column vector is the product of the two single-row
vectors: its nine finite coordinates are \(36s_0+s_1\), with
\(s_0,s_1\in\{9,16,17\}\), and their costs are the sums of the
corresponding entries \((2,0,2)\).

The joint certificate has \(11287\) nodes, \(86636\) labelled
edges and \(2098\) terminal nodes. As before, a context
\((b,\ell)\) records the input balance and the number of interior
columns read, capped at four. Each node carries a literal extended-integer
vector \(V\). Exact verification checks the following assertions:
\begin{enumerate}
 \item The root has context \((-1,0)\), dominates the exact
       post-first-column vector, and reaches every certificate node.
 \item There is exactly one outgoing edge for every allowed letter
       \((x,q^0,q^1)\). The successor context is
       \((b+2x-1,\min\{4,\ell+1\})\), and an edge is allowed
       precisely when the new balance lies in \([-3,3]\).
 \item Every edge \(v\xrightarrow{(x,q^0,q^1)}v'\) satisfies
       \(V_{v'}\ge P_{b_v,x,q^0,q^1}(V_v)\) in all coordinates.
 \item At every node with \(\ell=4\) and \(b=-1\),
       \(\mathcal P_{-1}(V)\) is finite and at most six.
\end{enumerate}
The terminal values range from \(-4\) to six. The checker constructs
the synchronised transitions from a separate count-equation encoding
of the row geometry. All arithmetic is exact; omitted coordinates
represent \(+\infty\), and no finite costs are discarded or shifted.

For any even \(w\ge6\) and any input in
Lemma~\ref{lem:pair6}, follow the certificate edges for its
\(w-2\) interior letters. The true minimum-cost vector is bounded
above by the certificate vector at the root. Monotonicity and the
third check preserve this relation at each letter. Since
\(2k=w-2\), the final interior context is \((-1,4)\). The
fourth check therefore gives a finite accepting cost at most six,
which proves \eqref{eq:pair6} for every even width.

For the deterministic construction, retain a minimising predecessor
at each state and column and backtrack from the accepting state.
Fixed enumeration resolves ties. The state space and the transition
weights are bounded independently of \(w\), so finite accumulated
costs have \(O(\log w)\) bits. Time and space are
\(O(w\log w)\). As in the single-row proof, the output is chosen
after both input words are known; the certificate does not assert an
online choice of routes.

\subsection{The standard initial interface}
\label{app:first3}

For Lemma~\ref{lem:first3}, the input is \(s_w=0(10)^k0\),
with \(k\ge2\). Use the same \(36\)-state row model and the
operators \(T_{b,x,q}\). After the true first zero column, its
exact cost vector \(V^0\) has
\[
 V^0_9=2,\qquad V^0_{16}=0,\qquad V^0_{17}=2,
\]
with all other entries infinite. An interior \(10\) block begins
at balance \(-1\), passes through balance zero and returns to
\(-1\). For its two difference bits \((q_1,q_0)\), define
\begin{equation}
 U_{q_1,q_0}=T_{0,0,q_0}\circ T_{-1,1,q_1}.
 \label{eq:first3-block}
\end{equation}
Both column transitions retain their original conditions, including
the intermediate output balance and the incidence bound. The grouping
changes only the way the fixed input pattern is read.

The certificate has \(43\) nodes and \(172\) edges. A node
contains a full \(36\)-coordinate vector and a count
\(e\in\{0,1,2\}\) of completed blocks, capped at two. Its
independent checker verifies:
\begin{enumerate}
 \item The root has count zero and vector at least \(V^0\), and
       every node is reachable from it.
 \item Every node has exactly one edge for each of the four labels
       \((q_1,q_0)\), with target count \(\min\{2,e+1\}\).
 \item On each edge \(v\xrightarrow{(q_1,q_0)}v'\),
       \(V_{v'}\ge U_{q_1,q_0}(V_v)\) in every coordinate.
 \item At every count-two node, appending the true last zero gives
       \(\mathcal T_{-1}(V)\le3\), with a finite value.
\end{enumerate}
There are \(38\) terminal checks, with values between zero and
three. The comparisons use literal integers and infinity, without
normalising the vectors or omitting finite entries.

Given any \(k\ge2\) and any \(q\in\F^{2k+2}\), follow the
edge for each interior block. Monotonicity of
\eqref{eq:first3-block} proves inductively that the actual cost
vector is at most the current certificate vector. After \(k\)
blocks the count is two. The terminal check bounds the accepting cost
by three. Since \(D(s_w)=0\), this is exactly
\eqref{eq:first3}. The endpoint bits of \(q\) remain irrelevant
because their mask bits vanish. The original single-row optimiser
minimises this same objective over all legal outputs, so it produces
the required route with its existing \(O(w\log w)\) bounds.

\section{Exact verification and reproducibility}
\label{app:verification}

OpenAI Codex AI assistants were used to support literature searches, prepare the synthesis, and develop verification code. The author checked key sources and computational results. The paired-shear result in
Theorem~\ref{thm:paired} has an analytic proof; its finite tests check
the implementation and examples. Theorem~\ref{thm:main} depends on
the three exact finite certificates in Appendix~\ref{app:transport}.
The specialised table verification establishes
Lemma~\ref{lem:policies}, used in Theorem~\ref{thm:finitewidth}.
The tests of individual constructed matrices are distinct from these
proof-critical finite verifications.

The accompanying implementation uses the Python standard
library and requires Python 3.10 or later.
The seed routine uses a fixed pivot order and sets all free
variables to zero. The geometric routine outputs the occupied
coordinates, the sparse matrix, its local rules, and the complete
rectangular minor witness.
It does not use a polynomial table or an integer factor table
in the supplied-polynomial synthesis.

For each cell, the source offsets are ordered as
\[
 \mathrm{NW},\mathrm{N},\mathrm{NE},\mathrm{W},\mathrm{C},
 \mathrm{E},\mathrm{SW},\mathrm{S},\mathrm{SE}.
\]
The first bit, \(\mathrm{NW}\), is the least significant bit
of the integer mask. The central bit records self-dependence;
offsets outside the occupied region have coefficient zero.
The data store the diagonal word in serpentine order and
matrix indices in row-major order.

The original full-matrix verifier imports neither the seed generator nor
the shear generator. It checks every nonzero matrix entry for locality
and verifies strong connectivity using forward and reverse reachability
in the dependency graph. It also
checks rule masks, disjointness and internal connectivity of
branch sets, and every edge between adjacent branch sets.
Its characteristic-polynomial computation uses the
division-free Berkowitz recurrence~\cite{B84}, independently
of the path determinant recurrence.
For a block matrix
\[
 M=\begin{pmatrix}a&r\\ c&D\end{pmatrix},
\]
the recurrence multiplies the coefficient vector of the characteristic
polynomial of \(D\) by the lower-triangular Toeplitz matrix with
first column \((1,a,rc,rDc,\ldots)^{\mathsf T}\).
In characteristic two, the usual minus signs are plus signs.
The identity follows by expanding
\((xI-D)^{-1}=\sum_{j\ge0}D^j x^{-j-1}\) in the Schur
determinant formula and comparing the polynomial parts.

\subsection{Checks for the paired shear}

The paired-shear test suite has the following scope:
\begin{enumerate}
 \item All \(256\) and \(512\) binary diagonal words are checked on
       \(\Gamma_{2,4,0}\) and \(\Gamma_{2,4,1}\), respectively.
       These auxiliary regions have eight and nine cells.
 \item For every \(36\le N\le255\), four fixed words are checked
       on \(\Gamma_N\): the zero word, the all-one word, the physical
       row-parity word, and the word whose path-index bit at \(t\) is
       \(\lfloor(17t^2+13t+5)/4\rfloor\bmod2\).
       This gives \(880\) geometric outputs.
 \item The three supplied inputs in Table~\ref{tab:examples} are
       checked separately, including irreducibility and the full-matrix
       characteristic polynomial.
\end{enumerate}
For all \(1651\) instances, the test suite checks similarity over
\(\F\), retention of the spanning path, the matching property of the
shear pairs, the degrees and edge counts, and the full rectangular minor.
The \(880\) geometric tests use
similarity rather than a separate characteristic-polynomial computation;
the remaining \(771\) instances recompute that polynomial by the
division-free verifier. The eight- and nine-cell cases check boundary
behaviour and do not extend Theorem~\ref{thm:main} below \(N=36\).
The recorded results are in \texttt{CHECKS.json}.

For each of the three examples, irreducibility is checked by
\[
 x^{2^n}\equiv x\pmod p,\qquad
 \gcd(x^{2^{n/q}}-x,p)=1
 \quad\text{for every prime }q\mid n.
\]
The first congruence forces \(p\) to be square-free and every
irreducible factor degree to divide \(n\).
If \(p\) were reducible, one such degree would be a proper
divisor of \(n\), and would divide \(n/q\) for some prime
\(q\mid n\), contradicting the corresponding greatest
common divisor test. Primitivity is then checked by
\eqref{eq:ordertest}, with the complete prime factorisation
of \(2^n-1\); the primality of the listed factors is also
checked exactly.

The supplement contains the generators, independent verifiers, example
data and exact arithmetic records. From a writable copy, the paired-shear
checks can be repeated with
\begin{verbatim}
python3 -B seed.py --self-test
python3 -B check_suite.py
python3 -B check_examples.py
python3 -B verify_all.py example-N36.json
python3 -B verify_all.py example-N37.json
python3 -B verify_all.py example-N64.json
\end{verbatim}
Before running \texttt{check\_suite.py} in the writable copy,
remove the following four files:
\begin{verbatim}
CHECKS.json
certificate-N36.json
certificate-N37.json
certificate-N64.json
\end{verbatim}
The script refuses to overwrite existing output files. Running it in
the writable copy preserves the delivered records separately.
The seed and original matrix-verifier files have SHA-256 hashes
\begin{verbatim}
seed.py
a79bb00c65c10d4058009aea62bc14cfcd4c69f25fa910f2892b8e73797558d3
verify_all.py
61e38ba5fef1907871aee732d3f2d42b860a273a3c8305b99e9712e0a7ac83c2
\end{verbatim}
The supplement's \texttt{SHA256SUMS} records every delivered file.

\subsection{The finite-width certificate}

The proof-critical finite data for Lemma~\ref{lem:policies} are in
\texttt{TABLES.json}, in the supplement's \texttt{fixed-width} directory.
The file's SHA-256 hash is
\begin{verbatim}
50653e742db5efbc826407b69757a1fdfed9797d1781abd79796a228d9f15898
\end{verbatim}
The certificate lists all states in lexicographic order, the indices of
the admissible states, the initial states, their potential values and
every policy row. A row records the current state index, internal word,
mask, outgoing tuple index, successor index and cost. Bit position zero, the least significant bit of the stored word or
mask, represents physical column one.
The independent checker reconstructs all available links and verifies
the complete domains and every obligation in Lemma~\ref{lem:policies}.
From the \texttt{fixed-width} directory, run
\begin{verbatim}
python3 -B verify.py --tables
python3 -B run_checks.py
\end{verbatim}
The first command must return \texttt{PASS}, with exactly \(36,464\)
verified policy rows. This finite calculation is a premise of Theorem~\ref{thm:finitewidth}.
The theorem's proof uses the potential inequality and the grid-minor
construction to extend the verified properties to every height.

The second command regenerates a separate integration record and
\(27\) JSON output files. They cover the three supplied inputs of
degrees \(36,37,64\); three boundary cases at each physical width
\(6,\ldots,11\), with \((h,r)=(2,0),(2,1),(3,w-1)\); and the
six canonical endpoint sizes \(48,63,80,99,120,143\).
For the three boundary cases at each width, the diagonal words are,
respectively, the zero word, the all-one word and the physical
row-parity word. At each endpoint size, the diagonal bit at cell
\((i,j)\) is one exactly when \((5i+3j+ij)\bmod7\in\{0,1,4\}\).
All \(27\) outputs have independent similarity, degree, rule and
full-minor checks. Ten full-matrix characteristic polynomials are
recomputed: the three supplied inputs, the six endpoint outputs and
the auxiliary \(\Gamma_{2,6,0}\) output. The \(24\)
direct-diagonal cases carry no irreducibility or primitivity claim.
These integration tests check the implementation; the exhaustive policy
verification, rather than extrapolation from these outputs, supplies
the finite premise of the theorem.

\subsection{The all-width transport certificate}

The directory \texttt{transport} contains the single-gap construction,
its independent checker based on sets of matrix entries, and the finite
certificate used in Appendix~\ref{app:transport}. The file \texttt{ft/cover.json} has
SHA-256 hash
\begin{verbatim}
8a9d5799230c3317331dde9076a7f72643bd6773c70d4e122c39509774548dc3
\end{verbatim}
It stores the \(277\) contexts and cost vectors, the \(1072\)
labelled transitions, and the designated root. The \(94\) terminal
inequalities are checked at the nodes with the terminal contexts specified
in Appendix~\ref{app:transport}. From the \texttt{transport}
directory, run
\begin{verbatim}
python3 -B ft/verify.py
python3 -B verify_all.py
\end{verbatim}
The first command returns \texttt{PASS\_COMPLETE\_FT\_CERTIFICATE}.
It compares all \(4032\) local transition rows in two encodings,
then checks the initial vector, every certificate transition and every
eligible terminal. The maximum terminal value is four. The fixed
local potential contribution is \(v|b+1|\), for prefix balance
\(b\) and current bit \(v\); it is independent of width and of
the preceding bit. The proof of Lemma~\ref{lem:transport-row} uses
the verified inequalities and monotone induction, rather than an
extrapolation from tested widths.

As an implementation check, it also compares the dynamic program with
direct endpoint enumeration on all \(6400\) allowed inputs at widths
six, seven and eight. The included general-potential normalisation
calculation is not a premise of the transport bound.

The second command returns \texttt{PASS\_A2\_ASSEMBLY\_CHECKS}.
It repeats the finite certificate verification and checks \(192\)
general-layout outputs, \(224\) diagonal words on \(56\) canonical
sizes, and the three supplied polynomial examples. The independent
checker reconstructs matrix products using sets and symmetric differences,
checks every final support entry and degree, and verifies the path,
local rules and full minor certificates. It also checks the public
synthesis command and rejects three altered certificates. The recorded
output is \texttt{transport/CHECKS.json}; the three example matrices
are in \texttt{transport/examples}.

The constructor checks the identities of its fixed dependencies before
using them. Script time and memory limits govern reproduction and do
not restrict the theorem. The geometric routine accepts every binary
path diagonal; polynomial synthesis assumes irreducibility.

\subsection{Joint routing and the strict connection bound}
\label{app:strict-reproduction}

The directory \texttt{strict3n} contains the construction for
Theorem~\ref{thm:main}, its independent matrix checker and the
standard-start certificate. Its subdirectory \texttt{baseline\_even3n}
contains the joint routing algorithm and the exact certificate for
Lemma~\ref{lem:pair6}, together with the unchanged single-row
dependencies. These source files are retained byte for byte.
From \texttt{strict3n}, run
\begin{verbatim}
python3 -B baseline_even3n/pair/verify.py
python3 -B check_first3.py
python3 -B verify_all.py
python3 -B construction.py --polynomial 0x1000000801 \
    --output automaton.json
\end{verbatim}
The first command returns \texttt{PASS\_PAIR6\_CLOSED\_COVER} and
checks all \(11287\) nodes, \(86636\) edges and \(2098\)
terminal obligations of Appendix~\ref{app:pair6}. Its compressed
certificate \texttt{baseline\_even3n/pair/cover.jsonl.gz} has
SHA-256 hash
\begin{verbatim}
6e02b2e16d1489f815c6b2f0589cd805ef10827da450a04b29986e0c5a729d28
\end{verbatim}
The storage format records exact signed costs, with unlisted
coordinates interpreted as infinity. The checker reserves an integer
marker for infinity, validates every finite entry, and performs
transition arithmetic using arbitrary-precision integers. Compression
is lossless and changes none of the inequalities.

The second command returns \texttt{PASS\_FIRST3\_CLOSED\_COVER}
and checks all \(43\) nodes, \(172\) edges and \(38\) terminals
in Appendix~\ref{app:first3}. The file \texttt{first3-cover.json}
has SHA-256 hash
\begin{verbatim}
c3c7a102a1338113740b867b46a1cd23ee53bcea1ee4556b3dc8919df4f0014c
\end{verbatim}
It stores all vector coordinates, using JSON null for infinity. The
checker composes the independently encoded column transitions in the
order specified by \eqref{eq:first3-block}.

The third command checks the construction and returns
\begin{verbatim}
PASS_STRICT3N_NEW_INTERFACES
\end{verbatim}
It compares the standard-start optimiser with direct ordered-position
enumeration on \(320\) inputs at widths six and eight, or \(80\)
after irrelevant endpoint-bit duplicates are removed. It also checks
\(24\) boundary-family diagonals at sizes
\(48,80,120,168,224,288\), \(64\) instances of the other branches,
the three supplied polynomials and the public synthesis command.
The independent checker reconstructs the simultaneous similarity,
counts the actual support after cancellation, and verifies all degree,
path, local-rule and full-minor requirements. For the boundary-family
instances, it checks that gap zero is treated alone and the remaining gaps are grouped in consecutive
pairs, and rejects an incorrect grouping, a deleted
path direction and a damaged certificate terminal. The recorded
results are in \texttt{strict3n/CHECKS.json}, with the separate joint
certificate check in \texttt{strict3n/PAIR6-CHECKS.json}.
These finite construction tests check implementation correspondence;
the unrestricted result uses the inductions in
Appendix~\ref{app:transport} and the exact budgets in
Section~\ref{sec:transport}.

All synthesis routines assume that the supplied polynomial is
irreducible. The fixed primitive examples have separate arithmetic
certificates. The implementation does not include Shoup's
irreducible-polynomial construction used in Corollary~\ref{cor:factors},
and no running-time benchmark is reported for the certified-factor
algorithm. The finite-width tables cover only \(6\le w\le11\).

\section*{Code and data availability}
The companion package~\cite{ChenCertificates}, version v010, contains
the programs, exact finite certificates and fixed example outputs used
in this manuscript. It includes the paired-shear and finite-width
constructions, the single-gap transport certificate, and the joint
two-gap and standard-start certificates for the strict connection bound,
together with their independent checkers and complete fixed dependency
trees. The README files give the reproduction commands, and the archive
manifest records the SHA-256 hashes of the files. The LaTeX source includes
the TikZ code for all three schematic figures.

\section*{Declaration of generative AI and AI-assisted technologies}
OpenAI Codex was used to assist with manuscript organisation, literature searches, and revision of the English prose. Its use in the research is described in Appendix~\ref{app:verification}. The author reviewed the manuscript, references, and computational results and takes full responsibility for the content of the article.

\end{document}